\documentclass[11pt]{article}

\usepackage[T1]{fontenc}
\usepackage{lmodern}
\usepackage[margin=1.08in]{geometry}
\usepackage{amsmath,amssymb,amsthm,mathtools}
\usepackage{microtype}
\usepackage{xcolor}
\usepackage{enumitem}
\usepackage{tikz}
\usepackage{float}
\usepackage{hyperref}

\hypersetup{
  colorlinks=true,
  linkcolor=blue!55!black,
  citecolor=blue,
  urlcolor=blue!65!black,
  pdftitle={Zonoids whose polars are zonoids: the Banach--Mazur distance need not tend to one},
  pdfauthor={Dmitry Ryabogin and Artem Zvavitch}
}

\mathtoolsset{showonlyrefs}

\newtheorem{theorem}{Theorem}[section]

\newtheorem{lemma}[theorem]{Lemma}
\newtheorem{corollary}[theorem]{Corollary}
\theoremstyle{definition}

\theoremstyle{remark}
\newtheorem{remark}[theorem]{Remark}

\newcommand{\R}{\mathbb R}
\newcommand{\Sph}{\mathbb S}
\newcommand{\dd}{\,\mathrm d}
\newcommand{\E}{\mathbb E}
\newcommand{\abs}[1]{\left|#1\right|}
\newcommand{\norm}[1]{\left\|#1\right\|}
\newcommand{\ip}[2]{\left\langle #1,#2\right\rangle}
\newcommand{\BM}{\mathrm{BM}}
\newcommand{\erf}{\operatorname{erf}}

\pgfmathdeclarefunction{tikzerf}{1}{%
  \pgfmathparse{1/(1+0.3275911*abs(#1))}%
  \edef\tikzerftmp{\pgfmathresult}%
  \pgfmathparse{(#1<0 ? -1 : (#1>0 ? 1 : 0))
    *(1-(((((1.061405429*\tikzerftmp-1.453152027)*\tikzerftmp
    +1.421413741)*\tikzerftmp-0.284496736)*\tikzerftmp
    +0.254829592)*\tikzerftmp)*exp(-(#1)^2))}%
}

\title{Zonoids whose polars are zonoids:\\
the Banach--Mazur distance need not tend to one}
\author{Dmitry Ryabogin and Artem Zvavitch%
\thanks{The authors are supported in part by U.S. National Science Foundation
Grants DMS-2247771 and DMS-2604412.}}
\date{}

\begin{document}

\maketitle

\begin{abstract}
For every $n\geq2$, we consider a  Gaussian
zonoid of revolution $Z_n$ arising from  works of Vitale and Mathis.  We prove that
$Z_n^\circ$ is also a zonoid and compute the Banach--Mazur distance
from $Z_n$ to the Euclidean ball.  This distance is independent of the
dimension and is approximately $1.10$.  Consequently, the supremal
Banach--Mazur distance among zonoids whose polars are zonoids does not
converge to $1$.   The construction also produces a
separable real Banach space $X$, not isometric to a Hilbert space, such
that both $X$ and $X^*$ embed linearly isometrically into $L_1$.
\end{abstract}



\tableofcontents

\section{Introduction}

A \emph{zonotope} is a Minkowski sum of finitely many line segments.
In the origin-symmetric setting considered in this note, it can be
written as
\[
 Z=\sum_{j=1}^N[-v_j,v_j].
\]
An origin-symmetric convex body $Z\subset\R^n$ is called a
\emph{zonoid} if it is a Hausdorff limit of zonotopes.  The class of
zonoids is rich: it contains every origin-symmetric planar convex body
and is closed under linear images and Minkowski addition.  In
dimensions $n\geq3$, however, it is not closed under polarity.  For
example, the cube $[-1,1]^n$ is a zonoid, whereas its polar, the
cross-polytope, is not. 
 In this paper we study zonoids whose polars are also zonoids.

There is a direct functional-analytic interpretation of this
question.  Let $X_Z$ be the normed space whose unit ball is
$Z^\circ$, the polar body of $Z$.  Under the usual identification of $\R^n$ with its dual,
the unit ball of $X_Z^*$ is then $Z$.  The classical correspondence
between zonoids and finite-dimensional subspaces of $L_1$ gives
\[
 Z\text{ and }Z^\circ\text{ are zonoids}
 \quad\Longleftrightarrow\quad
 X_Z\text{ and }X_Z^*\text{ admit linear isometric embeddings
 into }L_1.
\]
Moreover, $Z$ is an ellipsoid if and only if the norm of $X_Z$ is
Hilbertian.  Thus the study of zonoids whose polars are also zonoids
is the geometric form of a classical isometric problem in Banach
space theory.  We recall the details of this correspondence in
Subsection~\ref{subsec:zonoids-L1}.

Bolker asked whether, in dimension at least $3$, the simultaneous
assumptions that $Z$ and $Z^\circ$ are zonoids force $Z$ to be an ellipsoid
\cite[Conjecture~6.8]{Bolker1969}.  Schneider disproved this in every
dimension $n\geq3$ by a spherical-harmonic perturbation of
the Euclidean ball \cite{Schneider1975}.  Schneider's paper does not
provide a dimension-independent lower bound for the distance from the
Euclidean ball.  Lonke later constructed nonsmooth examples in
dimensions $3$ and $4$ and proved that a zonoid whose polar is also a zonoid cannot have a proper
Minkowski summand of dimension between $1$ and $n-2$
\cite{Lonke1997}.  In particular, the known low-dimensional examples
cannot be lifted by the most direct Minkowski-sum construction.

The asymptotic question remained: must all such examples become
arbitrarily close to the Euclidean ball as the dimension tends to
infinity?  This question appears as Problem~1.02, attributed there to
Rolf Schneider, on the AIM list arising from the 2007 workshop on
Fourier analytic methods in convex geometry \cite{AIM2007}.  According
to the authors' recollection, Joram Lindenstrauss also repeatedly
promoted the broader problem of zonoids whose polars are also zonoids and emphasized its isometric
functional-analytic formulation.  His involvement is also reflected
in Lonke's acknowledgment of discussions with Lindenstrauss
\cite[p.~12]{Lonke1997} and in Pisier's account
\cite[discussion following Corollary~2.7]{Pisier2012}.

For $n\geq1$, put
\begin{equation}\label{eq:Dn-definition}
 D_n:=\sup\left\{
 d_{\BM}(Z,B_2^n):
 Z\subset\R^n,\ Z\text{ and }Z^\circ\text{ are zonoids}
 \right\},
\end{equation}
where $d_{\BM}(Z,B_2^n)$ denotes the Banach--Mazur distance between $Z$ and the Euclidean unit ball $B_2^n\subset\R^n$.
  The question is whether $D_n\to1$.  Notice that Grothendieck's theorem
already gives the dimension-free upper bound
\begin{equation}\label{eq:Grothendieck-upper}
 D_n\leq K_G^{\R}<1.79,
\end{equation}
where $K_G^{\R}$ denotes the real Grothendieck constant. The above estimate follows from Krivine
\cite{Krivine1979}; see Pisier's quantitative formulation
\cite[Corollary~2.7]{Pisier2012}.  Strict improvements of Krivine's
bound are known; see \cite{BravermanMakarychevMakarychevNaor2011} and
the recent quantitative improvement \cite{SahaEtAl2026}.  These
improvements do not change the two-decimal estimate used here.
Thus the issue is asymptotic
roundness, rather than boundedness.

In this paper we prove that $D_n$ is bounded away from $1$ uniformly
in $n\ge 2$, and hence that $D_n\not\to1$.  The examples establishing this
result arise from Gaussian construction.  Let
\[
 \Gamma_{n-1}=(\gamma_1,\ldots,\gamma_{n-1}),
\]
where $\gamma_1,\ldots,\gamma_{n-1}$ are independent centered Gaussian
random variables with variance $\pi/2$.  Equivalently,
$\Gamma_{n-1}$ is a centered Gaussian vector in $\R^{n-1}$ with
covariance $(\pi/2)I_{n-1}$, where $I_{n-1}$ denotes the
$(n-1)\times(n-1)$ identity matrix.  This normalization is chosen so
that
$
 \E|\gamma_j|=1$. 
 
 For a convex body $K\subset\R^n$, let
$
 h_K(u):=\max_{y\in K}\ip{u}{y}
$
denote its support function.  Define
$Z_n\subset\R^{n-1}\times\R$ by prescribing its support function:
\begin{equation}\label{eq:Zn-intro}
 h_{Z_n}(x,t)
 =\E\abs{\ip{x}{\Gamma_{n-1}}+t}.
\end{equation}
Indeed, the right-hand side is an average of support functions of
origin-symmetric line segments and is therefore the support function
of a zonoid.  The chosen Gaussian normalization ensures that
$
 h_{Z_n}(x,0)=\abs{x},
$
while
$
 h_{Z_n}(0,t)=\abs{t}
$
follows immediately from the definition.
The body $Z_n$ is invariant under $O(n-1)$ acting on the
first $n-1$ coordinates; equivalently, it is a body of revolution
about the last coordinate axis.  For any body $K\subset\R^{n-1}\times\R$
with this invariance, we call the one-variable function
\[
 \varphi_K(s):=h_K(e_1,s),\qquad s\in\R,
\]
where $e_1\in\R^{n-1}$ is the first coordinate vector, its \emph{support
profile}.  This is simply convenient terminology for
the restriction of the support function to a two-dimensional meridian:
by rotational invariance and homogeneity,
\[
 h_K(x,t)=\abs{x}\,\varphi_K(t/\abs{x}),\qquad x\neq0.
\]
For $K=Z_n$, this reconstruction from the support profile gives exactly
the support function of the  Gaussian zonoid of Mathis and Vitale; see \cite[Theorem~3.1]{Vitale1991} and \cite[equation~(2.12) and the proof of
Theorem~2.10]{Mathis2025}. 

The zonoid $Z_n$ was introduced by Mathis, who also obtained a uniform upper bound on $d_{\BM}(Z_n,B_2^n)$. We prove that its polar $Z_n^\circ$ is also a zonoid and
determine the exact Banach--Mazur distance in
Theorem~\ref{thm:exact-BM}. In particular, we obtain the
following result.

\begin{theorem}\label{thm:main-finite}
For every $n\geq2$, the body $Z_n$ defined by
\eqref{eq:Zn-intro} and its polar $Z_n^\circ$ are zonoids.  Moreover,
\begin{equation}\label{eq:main-distance}
 d_{\BM}(Z_n,B_2^n)
 \ge 1.09.
\end{equation}
\end{theorem}



The same dimension-free profile has an infinite-dimensional consequence.

\begin{theorem}\label{thm:main-infinite-intro}
There is a separable infinite-dimensional real Banach space $X$,
linearly isomorphic but not linearly isometric to a Hilbert space,
such that both $X$ and $X^*$ embed linearly isometrically into
$L_1([0,1])$.
\end{theorem}

Grothendieck asked whether the two isometric-embedding assumptions in
Theorem~\ref{thm:main-infinite-intro} force the norm to be Hilbertian.
Pisier records that the infinite-dimensional real and complex cases
were open \cite[p.~249, immediately after
Corollary~2.7]{Pisier2012}; see also \cite[p.~1]{Lonke1997}.
Thus our Theorem~\ref{thm:main-infinite-intro} answers the real isometric
problem negatively.  In response to prompts from the authors based on ideas developed in
the present real argument, OpenAI's ChatGPT subsequently produced a
complete solution to the finite-dimensional complex problem.  We will record this result in a companion
arXiv note \cite{RyaboginZvavitchComplex}, which is not intended for
separate journal publication, see Subsection \ref{complex} for more details.

The paper is organized as follows.
Section~\ref{sec:preliminaries} recalls the necessary background.
Section~\ref{sec:gaussian-zonoid} studies $Z_n$ and its polar and
describes their meridian geometry.
Section~\ref{sec:polar-zonoid} proves that $Z_n^\circ$ is a zonoid.
Section~\ref{sec:exact-BM} determines the exact Banach--Mazur distance.
 Section~\ref{sec:infinite-dimensional} gives the
infinite-dimensional application. We conclude in  Section~\ref{sec:rem} with remarks on intersection bodies and the complex
analogue.

\section{Preliminaries}\label{sec:preliminaries}

\subsection{Support functions, polarity, and zonoids}\label{subsec:zonoids-L1}

All convex bodies are compact, convex, origin-symmetric, and have
nonempty interior, unless stated otherwise.  For a convex body
$K\subset\R^n$, containing the origin,  its support function, Minkowski functional, and polar
are
\[
 h_K(u)=\max_{x\in K}\ip{u}{x},
 \qquad
 \norm{u}_K=\inf\{a>0:u\in aK\},
 \qquad
 K^\circ=\{u:h_K(u)\leq1\}.
\]
The elementary dual identities
\begin{equation}\label{eq:polar-identities}
 h_{K^\circ}(u)=\norm{u}_K,
 \qquad
 \norm{u}_{K^\circ}=h_K(u)
\end{equation}
will be used repeatedly.

For origin-symmetric bodies $K,L\subset\R^n$, their
Banach--Mazur distance is
\begin{equation}\label{eq:BM-definition}
 d_{\BM}(K,L)=\inf\left\{a\geq1:
 L\subset TK\subset aL\text{ for some }T\in GL(n)\right\}.
\end{equation}
In particular,
\begin{equation}\label{eq:BM-ellipsoids}
 d_{\BM}(K,B_2^n)
 =\inf_E\inf\{a\geq1:E\subset K\subset aE\},
\end{equation}
where $E$ runs over origin-centered ellipsoids.  Polarity preserves
the distance:
\begin{equation}\label{eq:BM-polar}
 d_{\BM}(K,L)=d_{\BM}(K^\circ,L^\circ).
\end{equation}
For standard background on the Banach--Mazur distance and its role in
finite-dimensional convexity, see
\cite[Definition~2.1.2, p.~49]{ArtsteinAvidanGiannopoulosMilman2015}
and \cite{TomczakJaegermann1989}.

A centrally symmetric \emph{zonotope} in $\R^n$ is a compact convex
set of the form
\begin{equation}\label{eq:zonotope-definition}
 P_N=\sum_{j=1}^N[-v_j,v_j],
 \qquad v_j\in\R^n.
\end{equation}
Its support function is
$h_{P_N}(u)=\sum_{j=1}^N\abs{\ip{u}{v_j}}$.  A centrally symmetric
zonotope is a convex body precisely when the vectors
$v_1,\ldots,v_N$ span $\R^n$.  A centrally symmetric
body is a \emph{zonoid} if it is a Hausdorff limit of such zonotopes. 
Equivalently, $Z\subset\R^n$ is a zonoid if and only if there is a
finite even positive Borel measure $\mu$ on $\Sph^{n-1}$ such that
\begin{equation}\label{eq:cosine-transform}
 h_Z(u)=\int_{\Sph^{n-1}}\abs{\ip{u}{v}}\,\dd\mu(v).
\end{equation}
See \cite[Section~3.5]{Schneider2014} and
\cite{Bolker1969,GoodeyWeil1993}.  A convenient probabilistic form is
the following: if $V$ is an integrable random vector in $\R^n$, then
\begin{equation}\label{eq:random-zonoid}
 h(u)=\E\abs{\ip{u}{V}}
\end{equation}
is the support function of a possibly lower-dimensional zonoid.  It
has nonempty interior precisely when $V$ is not almost surely
concentrated in a proper linear subspace.   For every
realization $v$ of $V$,
$u\mapsto\abs{\ip{u}{v}}=h_{[-v,v]}(u)$ is the support function of the
segment $[-v,v]$.  Consequently, \eqref{eq:random-zonoid} is the
support function of the Minkowski expectation $\E[-V,V]$, which is a
zonoid.

The following classical correspondence is the
functional-analytic form of the cosine-transform representation. Let $X$ be an $n$-dimensional real normed space.  Let
\[
 B_X:=\{x\in X:\norm{x}_X\leq1\}
\]
denote its closed unit ball, and let $X^*$ denote its dual space,
consisting of all continuous linear functionals on $X$, equipped with
the dual norm
\[
 \norm{f}_{X^*}
 :=\sup_{\norm{x}_X\leq1}\abs{f(x)}.
\]
After choosing a linear identification of $X$ with $\R^n$, we identify
$X^*$ with $\R^n$ through the standard pairing.  Under these
identifications,
$
 B_{X^*}=B_X^\circ.
$

An 
$n$-dimensional real normed space $X$, after a linear identification
with $\R^n$, embeds linearly isometrically into an $L_1$-space if and
only if
\[
 \norm{x}_X
 =\int_{\Sph^{n-1}}\abs{\ip{x}{u}}\,\dd\mu(u)
\]
for a finite even positive measure $\mu$, or equivalently if and only
if $B_{X^*}=B_X^\circ$ is a zonoid.  This is the case $p=1$ of the
classical L\'evy representation; see
\cite[p.~189]{BenyaminiLindenstrauss2000},
\cite[Section~6.1]{Koldobsky2005}, and, for the convex-geometric
formulation, \cite[Section~3.5]{Schneider2014}.

For a centrally symmetric body $Z\subset\R^n$, let $X_Z$ be the
normed space defined by
\[
 \norm{x}_{X_Z}=h_Z(x).
\]
Then $B_{X_Z}=Z^\circ$ and $B_{X_Z^*}=Z$, and therefore
\[
 \begin{split}
 Z\text{ is a zonoid}
 &\Longleftrightarrow X_Z\text{ embeds linearly isometrically into }L_1,\\
 Z^\circ\text{ is a zonoid}
 &\Longleftrightarrow X_Z^*\text{ embeds linearly isometrically into }L_1.
 \end{split}
\]
Consequently, the condition studied here is exactly that a
finite-dimensional normed space and its dual both embed isometrically
into $L_1$.  Taking adjoints and applying the Hahn--Banach theorem also
shows that an isometric embedding $X_Z^*\to L_1$ makes $X_Z$ a metric
quotient of an $L_\infty$-space; see
\cite[Theorem~3.3 and Chapter~4]{Rudin1991}.  Thus $X_Z$ is simultaneously an
isometric subspace of $L_1$ and a metric quotient of $L_\infty$; this
is the formulation used by Lonke \cite[p.~1]{Lonke1997}.

The quantitative finite-dimensional form of Grothendieck's theorem
then gives
\[
 d_{\BM}(Z,B_2^n)=d_{\BM}(X_Z,\ell_2^n)
 \leq K_G^{\R}<1.79.
\]
This formulation follows from work of
Lindenstrauss--Pe{\l}czy\'nski; see
\cite{Grothendieck1953,LindenstraussPelczynski1968} and
\cite[Corollary~2.7]{Pisier2012}.  This is an isomorphic conclusion:
it does not assert that the given norm of $X_Z$ is induced by an inner
product.  The corresponding isometric question is precisely whether
$Z$ must be an ellipsoid.

\subsection{Complete monotonicity and Bernstein functions}\label{compmonotone}
The following analytic notions enter through the proof that
$Z_n^\circ$ is a zonoid.  After using rotational symmetry to reduce
its support function to a one-variable function, we represent that
function as an integral, with respect to a positive measure, of
support functions of elementary zonoids.  The
Hausdorff--Bernstein--Widder theorem provides the required positive
measure once an associated function is shown to be completely
monotone.  Thus complete monotonicity converts our analytic
computations into the desired zonoid representation.

A $C^\infty$ function $f:(0,\infty)\to\R$ is \emph{completely monotone} if
\[
 (-1)^k f^{(k)}(u)\geq0
 \qquad \mbox{ for } k=0,1,2,\ldots \mbox{ and } u>0.
\]
If a completely monotone function $f$ has a finite limit at the
origin, we extend it continuously by setting
$
 f(0):=f(0+).
$
With this convention, we may also say that $f$ is completely monotone
on $[0,\infty)$.

The Hausdorff--Bernstein--Widder theorem says that $f$ is completely
monotone if and only if
\begin{equation}\label{eq:HBW}
 f(u)=\int_{[0,\infty)}e^{-u\lambda}\,\dd\nu(\lambda)
\end{equation}
for a positive measure $\nu$; if $f(0+)<\infty$, then
$\nu$ is finite.  A $C^\infty$ function
$B:(0,\infty)\to[0,\infty)$ is a \emph{Bernstein function}
if $B'$ is completely monotone.  The following standard facts will be
used:
\begin{enumerate}[label=(\roman*)]
\item products and pointwise limits of completely monotone functions
      are completely monotone whenever the limit is finite;
\item if $F$ is completely monotone and $B$ is a Bernstein function,
      then $F\circ B$ is completely monotone;
\item $B$ is a Bernstein function if and only if $e^{-tB}$ is
      completely monotone for every $t>0$.
\end{enumerate}
See \cite[Theorem~1.4, Corollary~1.6, and
Theorem~3.7]{SchillingSongVondracek2012}, as well as
\cite[Chapter~IV]{Widder1941}.

\section{The  Gaussian zonoid \texorpdfstring{$Z_n$}{Z\_n} and its polar}
\label{sec:gaussian-zonoid}
\subsection{The support profile of the Gaussian zonoid \texorpdfstring{$Z_n$}{Z\_n}}
Fix $n\geq2$ and identify
$\R^n=\R^{n-1}\times\R$.  Let $\Gamma_{n-1}$ have density
\begin{equation}\label{eq:Gamma-density}
 z\longmapsto \pi^{-(n-1)}
 \exp\!\left(-\frac{\abs{z}^2}{\pi}\right),
 \qquad z\in\R^{n-1}.
\end{equation}
Thus $\Gamma_{n-1}\sim N(0,(\pi/2)I_{n-1})$.  Define $Z_n$ by
\eqref{eq:Zn-intro}.  Formula \eqref{eq:random-zonoid}, applied to the
random vector $(\Gamma_{n-1},1)$, proves immediately that $Z_n$ is a
zonoid. It is full-dimensional. Indeed, otherwise there would exist a nonzero
$(x,t)\in\R^{n-1}\times\R$ such that
$
h_{Z_n}(x,t)=0.
$
By the definition of $Z_n$, this implies
$
\ip{x}{\Gamma_{n-1}}+t=0$ 
almost surely.
Taking variances and using the nondegeneracy of the Gaussian vector
$\Gamma_{n-1}$ gives
\[
0=\operatorname{Var}\bigl(\ip{x}{\Gamma_{n-1}}+t\bigr)
  =\operatorname{Var}\bigl(\ip{x}{\Gamma_{n-1}}\bigr)
  =\frac{\pi}{2}|x|^2,
\]
and hence $x=0$. Taking expectations then gives $t=0$, contradicting
the choice of $(x,t)$. Therefore $Z_n$ is full-dimensional.

We note that there is a useful geometric way to read this probabilistic
representation.  Let $\gamma_{n-1}$ be the distribution of
$\Gamma_{n-1}$ and, for $v\in\R^{n-1}$, let
$S_v=[-(v,1),(v,1)]$. 
 Since $h_{S_v}(x,t)=\abs{\ip{x}{v}+t}$, the identity
$h_{Z_n}(x,t)=\E\abs{\ip{x}{\Gamma_{n-1}}+t}$ says that
\begin{equation}\label{randomzon}
 Z_n=\int_{\R^{n-1}}S_v\,\dd\gamma_{n-1}(v),
\end{equation}
where the integral is understood in the sense of Minkowski integration:
the support function of the integral is the integral of the support
functions.  Thus one may picture a Gaussian cloud in the affine
hyperplane $x_n=1$, join each point $(v,1)$ to its antipode, and
average the resulting line segments.  The rotational invariance of
the Gaussian cloud makes $Z_n$ a body of revolution.   The
deterministic last coordinate distinguishes the axial direction from
the horizontal directions.  

We now compute the resulting support
function explicitly. We use the standard error function
\begin{equation}\label{eq:erf-definition}
 \erf(a):=\frac2{\sqrt\pi}\int_0^a e^{-r^2}\,\dd r,
 \qquad a\in\R.
\end{equation}
Let $\gamma\sim N(0,\pi/2)$, with density
$\pi^{-1}e^{-u^2/\pi}$.  Since $\Gamma_{n-1}$ is rotationally
invariant, for $x\in\R^{n-1}$, the random variables
$\ip{x}{\Gamma_{n-1}}$ and $|x|\gamma$ have the same distribution.
Consequently, if $x\not =0$, then
\[
 h_{Z_n}(x,t)
 =\E\abs{|x|\gamma+t}
 =|x|\E\abs{\gamma+\frac{t}{|x|}}.
\]
To calculate the last expectation, let $s\in\R$.  Splitting its
defining integral at $-s$ gives
\begin{equation}\label{eq:folded-normal}
 \E\abs{\gamma+s}
 =e^{-s^2/\pi}+s\erf\!\left(\frac{s}{\sqrt\pi}\right).
\end{equation}
This is the folded-normal first-moment formula; see also
\cite{LeoneNelsonNottingham1961}.  We therefore define the support
profile of $Z_n$ by
\begin{equation}\label{eq:phi-definition}
 \varphi(s)
 :=h_{Z_n}(e_1,s)
 =\E\abs{\gamma+s}
 =e^{-s^2/\pi}+s\erf\!\left(\frac{s}{\sqrt\pi}\right),
 \qquad s\in\R,
\end{equation}
where $e_1$ is the first coordinate vector in $\R^{n-1}$; the
function $\varphi$ is even.  Combining the preceding formulas yields,
for $r=\abs{x}>0$,
\begin{equation}\label{eq:Zn-profile}
 h_{Z_n}(x,t)=r\varphi(t/r)
 =r e^{-t^2/(\pi r^2)}
 +t\erf\!\left(\frac{t}{\sqrt\pi r}\right).
\end{equation}
At $r=0$, the continuous extension is $\abs t$.  This is exactly
Mathis's function 
\cite[equation~(2.12)]{Mathis2025}, after interchanging the axial and
horizontal variables.

Two elementary identities will be useful:
\begin{equation}\label{eq:phi-identities}
 \varphi'(s)=\erf\!\left(\frac{s}{\sqrt\pi}\right),
 \qquad
 \varphi(s)-s\varphi'(s)=e^{-s^2/\pi}.
\end{equation}
\subsection{The support profile of \texorpdfstring{$Z_n^\circ$}{Z\_n polar}}
We now derive the support profile of the polar directly from the
definition of the polar body.  Recall from Section~2.1 that
\[
 Z_n^\circ
 =\bigl\{(y,\tau)\in\R^{n-1}\times\R:
          h_{Z_n}(y,\tau)\leq1\bigr\}.
\]
Therefore, by the definition of the support function and its positive
homogeneity,
$$
 h_{Z_n^\circ}(e_1,s)=\sup_{h_{Z_n}(y,\tau)\leq1}
       \bigl(\ip{e_1}{y}+s\tau\bigr)=\sup_{(y,\tau)\neq(0,0)}
       \frac{\ip{e_1}{y}+s\tau}{h_{Z_n}(y,\tau)}.
$$
For $y\neq0$, write $y=r\theta$, where $r=\abs{y}>0$ and
$\theta\in\Sph^{n-2}$, and put $q=\tau/r$.  Formula
\eqref{eq:Zn-profile} then gives
\[
 \frac{\ip{e_1}{y}+s\tau}{h_{Z_n}(y,\tau)}
 =\frac{\ip{e_1}{\theta}+sq}{\varphi(q)}.
\]
For each fixed $q$, the numerator is maximized over
$\theta\in\Sph^{n-2}$ when $\theta=e_1$.  Thus the contribution of
all directions with $y\neq0$ is
\[
 \sup_{q\in\R}\frac{1+sq}{\varphi(q)}.
\]
The directions with $y=0$ contribute $\abs{s}$.  They are already
included in this supremum: by \eqref{eq:phi-definition},
$\varphi(q)/\abs q\to1$ as $\abs q\to\infty$, and the quotient above
tends to $\abs{s}$ by taking $q\to+\infty$ when $s\geq0$ and
$q\to-\infty$ when $s<0$.  Consequently, the support profile of
$Z_n^\circ$ is
\begin{equation}\label{eq:polar-profile-definition}
 \varphi^\circ(s)
 :=h_{Z_n^\circ}(e_1,s)
 =\sup_{q\in\R}\frac{1+sq}{\varphi(q)},
 \qquad s\in\R.
\end{equation}
Since $Z_n^\circ$ is also invariant under rotations of the first
$n-1$ coordinates, homogeneity gives
\begin{equation}\label{eq:polar-support-profile}
 h_{Z_n^\circ}(x,t)
 =\abs{x}\,\varphi^\circ(t/\abs{x}),
 \qquad x\neq0.
\end{equation}
Next we evaluate the supremum in \eqref{eq:polar-profile-definition} to obtain a closed form for $\varphi^\circ$. For $s,q\geq0$, differentiation of the quotient in
\eqref{eq:polar-profile-definition} gives
$$
 \frac{\dd}{\dd q}\frac{1+sq}{\varphi(q)}
 =\frac{s\varphi(q)-(1+sq)\varphi'(q)}{\varphi(q)^2}=\frac{s\bigl(\varphi(q)-q\varphi'(q)\bigr)
          -\varphi'(q)}{\varphi(q)^2}.
$$
For completeness, differentiating \eqref{eq:phi-definition} yields
$$
 \varphi'(q)
 =-\frac{2q}{\pi}e^{-q^2/\pi}
   +\erf\!\left(\frac q{\sqrt\pi}\right)
   +q\frac{2}{\pi}e^{-q^2/\pi}=\erf\!\left(\frac q{\sqrt\pi}\right),
$$
and therefore
\[
 \varphi(q)-q\varphi'(q)=e^{-q^2/\pi}>0.
\]
Thus a critical point satisfies
\begin{equation}\label{eq:polar-stationary}
 s=\frac{\varphi'(q)}{\varphi(q)-q\varphi'(q)}
 =e^{q^2/\pi}\erf\!\left(\frac q{\sqrt\pi}\right).
\end{equation}
This critical point is the unique global maximizer.  Indeed, set
\[
 S(q):=e^{q^2/\pi}\erf\!\left(\frac q{\sqrt\pi}\right).
\]
Then
\[
 S'(q)=\frac2\pi+\frac{2q}{\pi}S(q)>0,
 \qquad S(0)=0,
 \qquad \lim_{q\to\infty}S(q)=\infty.
\]
Moreover, the sign of the derivative of the quotient is the sign of
$s-S(q)$.  Hence the quotient increases up to the unique solution of
$S(q)=s$ and decreases thereafter.  Finally, when $s\geq0$ it is
enough to consider $q\geq0$, because $\varphi$ is even and replacing a
negative $q$ by $\abs q$ does not decrease the numerator $1+sq$.
At that point, \eqref{eq:polar-stationary} also gives
\[
 1+sq
 =1+\frac{q\varphi'(q)}{\varphi(q)-q\varphi'(q)}
 =\frac{\varphi(q)}{\varphi(q)-q\varphi'(q)}.
\]
Consequently,
\begin{equation}\label{eq:polar-profile-value}
 \varphi^\circ(s)
 =\frac1{\varphi(q)-q\varphi'(q)}
 =e^{q^2/\pi}.
\end{equation}
Equivalently, on putting $q=\sqrt\pi a$,
\begin{equation}\label{eq:polar-parametric}
 s=e^{a^2}\erf(a),
 \qquad
 \varphi^\circ(s)=e^{a^2},
 \qquad a\geq0.
\end{equation}
The first function in \eqref{eq:polar-parametric} is strictly
increasing from $0$ to infinity.

Finally, let $q=q(s)$ denote the unique maximizing point in
 \eqref{eq:polar-profile-definition}.  Since $s=S(q)$ and $S'(q)>0$,
the inverse-function theorem shows that $q(s)$ is smooth for $s>0$.
Differentiating
\[
 \varphi^\circ(s)=\frac{1+sq(s)}{\varphi(q(s))}
\]
and using the stationarity equation
$s\varphi(q)=(1+sq)\varphi'(q)$, the terms containing $q'(s)$ cancel,
and hence $(\varphi^\circ)'(s)=q/\varphi(q)$.  It follows that
\begin{equation}\label{eq:polar-intercept}
 \varphi^\circ(s)-s(\varphi^\circ)'(s)
 =\frac1{\varphi(q)}\longrightarrow0
 \qquad(s\to\infty),
\end{equation}
because $q(s)\to\infty$.  Moreover,
$\varphi(q)/q\to1$, and therefore
\begin{equation}\label{sloplimit}
\lim_{s\to\infty}(\varphi^\circ)'(s)=1.
\end{equation}
Thus $\varphi^\circ$ has limiting slope one.  Moreover, the tangent
line to its graph at $s$ has vertical-axis intercept
$
 \varphi^\circ(s)-s(\varphi^\circ)'(s),$
and \eqref{eq:polar-intercept} shows that this intercept tends to zero as
$s\to\infty$. These two limits
will be used when applying the zonoid criterion of
Lemma~\ref{lem:Gaussian-mixture} below: in the notation of that lemma,
$\ell=1$ and $c=0$, respectively.

\subsection{Geometric features of \texorpdfstring{$Z_n$ and $Z_n^\circ$}{Z\_n and polar}}
In this subsection, we discuss several geometric features of $Z_n$ and
$Z_n^\circ$ that help illuminate the geometry of this dual pair of
zonoids.  Although these observations are not used directly in the
proof, they provide useful geometric intuition for the strikingly
different behavior of $Z_n$ and $Z_n^\circ$ near their poles.

The preceding formulas allow us to describe several geometric
features of both $Z_n$ and its polar.  We begin with a general
geometric interpretation of the polar parametrization.  Consider a body of revolution $K$ with differentiable support profile
$\psi(s):=h_K(e_1,s)$, so that
\[
 h_K(x,t)=\abs{x}\,\psi(t/\abs{x}), \mbox{ for } x\not 0.
\]
Recall that, whenever $h_K$ is differentiable at $u\neq0$,
$\nabla h_K(u)$ is the unique boundary point of $K$ at which $u$ is
an outer normal; see \cite[Corollary~1.7.3]{Schneider2014}.  Writing $r=\abs{x}$ and $q=t/r$, differentiation of
$h_K(x,t)=r\psi(q)$ gives
\[
 \nabla_x h_K(x,t)
 =\bigl(\psi(q)-q\psi'(q)\bigr)\frac{x}{\abs{x}},
 \qquad
 \partial_t h_K(x,t)=\psi'(q).
\]
Consequently, for $\theta\in\Sph^{n-2}$, the boundary point of $K$
having outer normal $(\theta,s)$ is
\begin{equation}\label{eq:revolution-boundary-point}
 \nabla h_K(\theta,s)
 =\bigl((\psi(s)-s\psi'(s))\theta,\,\psi'(s)\bigr).
\end{equation}
Thus $\psi(s)-s\psi'(s)$ has two simultaneous geometric meanings.  It
is the vertical-axis intercept of the tangent line to the graph
of $\psi$ at $s$, and it is the horizontal radius of the corresponding
boundary point.  In particular, if
$\psi'(s)\to1$ and $\psi(s)-s\psi'(s)\to c$ as $s\to\infty$, then the
top supporting hyperplane meets the body in a disk of radius $c$.  The
case $c=0$ gives a singleton top exposed face rather than a facet.

This tangent-intercept description also makes the polar calculation
natural.  For $q\geq0$ and $\theta\in\Sph^{n-2}$, put
\[
 u(q,\theta):=\frac{(\theta,q)}{\varphi(q)}.
\]
Formula \eqref{eq:revolution-boundary-point} places $
 \nabla h_{Z_n}(\theta,q)
 =\bigl((\varphi(q)-q\varphi'(q))\theta,\varphi'(q)\bigr)
$
on $\partial Z_n$, while $h_{Z_n}(u(q,\theta))=1$ places
$u(q,\theta)$ on $\partial Z_n^\circ$.  Moreover,
\[
 \ip{\nabla h_{Z_n}(\theta,q)}{u(q,\theta)}
 =\frac{\varphi(q)-q\varphi'(q)+q\varphi'(q)}{\varphi(q)}=1.
\]
Consequently, $\nabla h_{Z_n}(\theta,q)$ is a supporting normal to
$Z_n^\circ$ at $u(q,\theta)$.  Dividing that normal by its horizontal
component $\varphi(q)-q\varphi'(q)$ gives the slope
\[
 s=\frac{\varphi'(q)}{\varphi(q)-q\varphi'(q)},
\]
and evaluating the support function of the polar in the normalized
normal direction $(\theta,s)$ gives
\[
 \varphi^\circ(s)
 =\ip{(\theta,s)}{u(q,\theta)}
 =\frac1{\varphi(q)-q\varphi'(q)}.
\]
This is the geometric reason for both
\eqref{eq:polar-stationary} and \eqref{eq:polar-profile-value}: the reciprocal of the vertical-axis intercept of the tangent line to
the support profile of $Z_n$ becomes the corresponding support value
of $Z_n^\circ$.

The same parametrization describes exactly what happens near the north and the south
poles.  Write $r$ for the horizontal radius and $z$ for the last
coordinate in a meridian plane.  Setting $q=\sqrt\pi a$ in
\eqref{eq:phi-identities} and
\eqref{eq:revolution-boundary-point} gives the upper
boundary of the meridian section of $Z_n$ as
\begin{equation}\label{eq:Zn-boundary-parametrization}
 r_{Z_n}(a)=e^{-a^2},\qquad z_{Z_n}(a)=\erf(a),\qquad a\geq0.
\end{equation}
Geometrically, the normal $(\theta,\sqrt\pi a)$ selects from almost every
segment $S_v$ in \eqref{randomzon}  the endpoint having the sign of $\ip{\theta}{v}+\sqrt\pi a$; averaging these selected endpoints gives
the point in \eqref{eq:Zn-boundary-parametrization}.  The standard
error-function tail estimate, obtained by integration by parts, is
\[
 1-\erf(a)
 =\frac{e^{-a^2}}{\sqrt\pi a}
  \bigl(1+O(a^{-2})\bigr).
\]
Eliminating $a$ by
$a=\sqrt{\log(1/r)}$ and writing the upper meridian as
$z_{Z_n}=z_{Z_n}(r)$, it follows that, as $r\downarrow0$,
\begin{equation}\label{eq:Zn-pole-asymptotic}
 1-z_{Z_n}(r)
 =\frac{r}{\sqrt{\pi\log(1/r)}}
  \left(1+O\!\left(\frac1{\log(1/r)}\right)\right).
\end{equation}
In fact,
\[
 z_{Z_n}'(r)=-\frac1{\sqrt{\pi\log(1/r)}}\longrightarrow0,
 \qquad
 z_{Z_n}''(r)
 =-\frac1{2\sqrt\pi\,r(\log(1/r))^{3/2}}
 \longrightarrow-\infty.
\]
For comparison, the unit circle satisfies
$1-z=r^2/2+O(r^4)$ near its north pole.  Thus $Z_n$ has a unique
horizontal tangent there and is $C^1$, but not $C^2$, at the pole; its
meridional curvature tends to infinity.  This accounts for the pointed
appearance in the figure, although the pole is not a conical vertex.

For the polar, set
\[
 \Phi(a):=\varphi(\sqrt\pi a)
 =e^{-a^2}+\sqrt\pi a\erf(a).
\]
The radial formula
$u(q,\theta)=(\theta,q)/\varphi(q)$, with $q=\sqrt\pi a$, gives
 the upper boundary of the meridian section of
$Z_n^\circ$:
\begin{equation}\label{eq:polar-boundary-parametrization}
 r_{Z_n^\circ}(a)=\frac1{\Phi(a)},
 \qquad
 z_{Z_n^\circ}(a)=\frac{\sqrt\pi a}{\Phi(a)},
 \qquad a\geq0.
\end{equation}
Since 
 $\Phi'(a)=\sqrt\pi\,\erf(a)>0,$ for $a>0$,
$r_{Z_n^\circ}(a)$ decreases from $1$ to $0$; hence $a$ can be
eliminated and the upper meridian can be written as a graph
$z_{Z_n^\circ}=z_{Z_n^\circ}(r)$.
The quantity by which the numerator in the second coordinate falls
short of its denominator is strictly positive:
$$
 \Phi(a)-\sqrt\pi a
 =e^{-a^2}-\sqrt\pi a\bigl(1-\erf(a)\bigr)=2\int_a^\infty (t-a)e^{-t^2}\,\dd t>0.
$$
Hence $z_{Z_n^\circ}(a)<1$ for every finite $a$: the top exposed face
is again a singleton. The next term in the
error-function tail expansion gives
\[
 \Phi(a)-\sqrt\pi a
 =\frac{e^{-a^2}}{2a^2}\bigl(1+O(a^{-2})\bigr).
\]
Since $r=r_{Z_n^\circ}(a)=1/\Phi(a)$, it follows that
\[
 r=\frac1{\sqrt\pi a}
 \left(1+O\!\left(\frac{e^{-a^2}}{a^3}\right)\right),
 \qquad
 \frac1{\pi r^2}=a^2+O\!\left(\frac{e^{-a^2}}a\right).
\]
Consequently,
$$
 1-z_{Z_n^\circ}(r)
 =\frac{\Phi(a)-\sqrt\pi a}{\Phi(a)}=\frac{e^{-a^2}}{2\sqrt\pi\,a^3}
   \bigl(1+O(a^{-2})\bigr)=\frac\pi2 r^3
   \exp\!\left(-\frac1{\pi r^2}\right)
   \bigl(1+O(r^2)\bigr).
$$
Thus
\begin{equation}\label{eq:polar-pole-asymptotic}
 1-z_{Z_n^\circ}(r)
 =\frac\pi2 r^3\exp\!\left(-\frac1{\pi r^2}\right)
  \bigl(1+O(r^2)\bigr)
 \qquad(r\downarrow0).
\end{equation}
The right-hand side is $o(r^N)$ for every $N>0$.  In this precise
sense the polar is flat to every algebraic order at its pole, which
explains why the drawing suggests a facet even though the top exposed
face is only a singleton.  Polarity therefore exchanges a rapidly
bending $C^1$ pole of $Z_n$ with an extremely flat singleton pole of
$Z_n^\circ$: the first pole is not a conical corner, and the second is
not a facet.

Figure~\ref{fig:meridian-sections} displays these two
dimension-independent meridian sections at the same scale.

\begin{figure}[H]
\centering
\begin{tikzpicture}[
  x=2.05cm,y=2.05cm,
  line cap=round,line join=round,
  every node/.style={font=\small,text=black!75!black}
]
\begin{scope}[xshift=-3cm]
  \path[fill=black!7,draw=black!75!black,line width=.75pt]
    (0,-1)
    plot[domain=-89.5:89.5,samples=181,variable=\b]
      ({exp(-tan(\b)^2/pi)},
       {tikzerf(tan(\b)/sqrt(pi))})
    -- (0,1)
    plot[domain=-89.5:89.5,samples=181,variable=\b]
      ({-exp(-tan(\b)^2/pi)},
       {-tikzerf(tan(\b)/sqrt(pi))})
    -- cycle;
  \draw[black!40,dashed,line width=.55pt] (0,0) circle[radius=1];
  \draw[black!45,line width=.35pt] (-1.08,0)--(1.08,0);
  \draw[black!45,densely dotted,line width=.45pt]
    (0,-1.08)--(0,1.08);
  \node[above] at (0,1.18) {$Z_n$};
  \node[font=\scriptsize,fill=white,inner sep=1pt]
    at (.72,.78) {$B_2^2$};
  \node[below] at (1.08,0) {$x_1$};
  \node[left] at (0,.92) {$x_n$};
\end{scope}

\begin{scope}[xshift=3cm]
  \path[fill=black!16,draw=black!75!black,line width=.75pt]
    (0,-1)
    plot[domain=-89.5:89.5,samples=181,variable=\b]
      ({1/(exp(-tan(\b)^2/pi)
          +tan(\b)*tikzerf(tan(\b)/sqrt(pi)))},
       {tan(\b)/(exp(-tan(\b)^2/pi)
          +tan(\b)*tikzerf(tan(\b)/sqrt(pi)))})
    -- (0,1)
    plot[domain=-89.5:89.5,samples=181,variable=\b]
      ({-1/(exp(-tan(\b)^2/pi)
          +tan(\b)*tikzerf(tan(\b)/sqrt(pi)))},
       {-tan(\b)/(exp(-tan(\b)^2/pi)
          +tan(\b)*tikzerf(tan(\b)/sqrt(pi)))})
    -- cycle;
  \draw[black!40,dashed,line width=.55pt] (0,0) circle[radius=1];
  \draw[black!45,line width=.35pt] (-1.08,0)--(1.08,0);
  \draw[black!45,densely dotted,line width=.45pt]
    (0,-1.08)--(0,1.08);
  \node[above] at (0,1.18) {$Z_n^\circ$};
  \node[below] at (1.08,0) {$x_1$};
  \node[left] at (0,.92) {$x_n$};
\end{scope}
\end{tikzpicture}
\caption{The dimension-independent meridian sections
$Z_n\cap\operatorname{span}\{e_1,e_n\}$ and
$Z_n^\circ\cap\operatorname{span}\{e_1,e_n\}$, drawn at the same
scale.  The dashed circle is the meridian section of $B_2^n$.}
\label{fig:meridian-sections}
\end{figure}
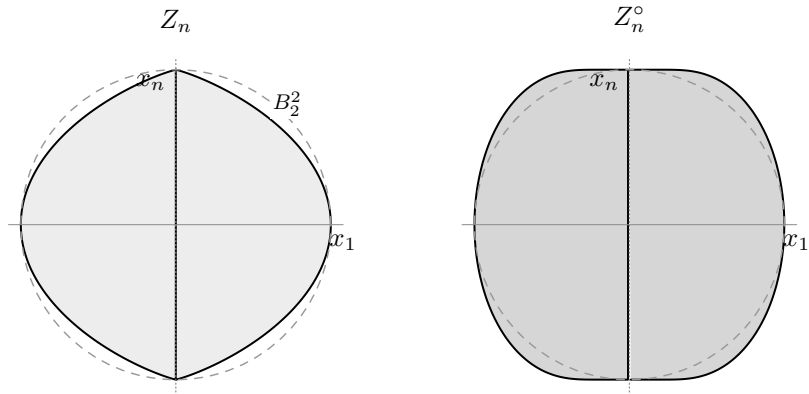

\section{The polar is a zonoid}
\label{sec:polar-zonoid}

Section~\ref{sec:gaussian-zonoid} gives an explicit formula
for the support profile $\varphi^\circ$ of $Z_n^\circ$.  In this section, we first
prove Lemma \ref{lem:Gaussian-mixture} which is a criterion ensuring that an even function is the support
profile of a zonoid of revolution.  The formulas obtained in
Section~\ref{sec:gaussian-zonoid} verify the remaining hypotheses of
this criterion, including the intercept condition
\eqref{eq:polar-intercept}; the substantive point left to prove is
that
\[
 u\longmapsto(\varphi^\circ)''(\sqrt u)
\]
is completely monotone.  We then establish this complete monotonicity and apply
Lemma~\ref{lem:Gaussian-mixture} to represent $h_{Z_n^\circ}$ as a
positive integral of support functions of Gaussian zonoids.  This
proves that $Z_n^\circ$ is itself a zonoid.

\subsection{A criterion for zonoids
of revolution}
\label{subsec:zonoid-criterion}

The representation of an even density as a positive integral of
centered Gaussian densities with varying variances---usually called a
Gaussian scale mixture---is classical.  Andrews and Mallows proved
that a density of the form $p(s)=f(s^2)$ admits such a representation
if and only if $f$ is completely monotone
\cite{AndrewsMallows1974}; see also \cite{West1987}.

The following lemma may be viewed as a support-function version of
this principle.  It converts complete monotonicity of a one-variable
function into a zonoid representation valid in every dimension.
Since the support profile is even, recovering it from its second
derivative introduces an additive constant.  The lemma includes a condition ensuring that this constant is
nonnegative.  Upon homogenization, the constant becomes the term
$c\abs{x}$, which is the support function of the Euclidean ball
$cB_2^{n-1}\times\{0\}$ in the equatorial subspace.  The new
ingredient needed for our construction is therefore the
complete-monotonicity result established in
Theorem~\ref{thm:complete-monotonicity}.

\begin{lemma}\label{lem:Gaussian-mixture}
Let $\psi:\R\to\R$ be even, convex, and of class $C^2$.  Assume, in
addition, that the function $u\longmapsto\psi''(\sqrt{u})$
is completely monotone  on $(0,\infty)$ and that the finite limit
\begin{equation}\label{eq:mixture-intercept}
 c:=\lim_{s\to\infty}\bigl(\psi(s)-s\psi'(s)\bigr)
\end{equation}
exists and is nonnegative. Then the limit
\[
 \ell:=\lim_{s\to\infty}\psi'(s)
\]
exists and is finite.  Moreover, for every $n\geq2$, the  function
\begin{equation}\label{eq:Hn-profile}
 H_n(x,t)=\abs{x}\psi(t/\abs{x}),
 \qquad (x,t)\in\R^{n-1}\times\R,\quad x\neq0,
\end{equation}
extended to the $t$-axis by
\[
 H_n(0,t)=\ell\abs{t},
\]
is the support function of a possibly lower-dimensional zonoid in
$\R^n$.  This zonoid has nonempty interior whenever $\ell>0$.

More precisely, there is a finite positive measure $\eta$ on
$(0,\infty)$ such that
\begin{equation}\label{eq:psi-Gaussian-mixture}
\psi(s)=c+\int_{(0,\infty)}\E\abs{\xi_\lambda+s}
 \,\dd\eta(\lambda),
\end{equation}
 where $\xi_\lambda$ has density
$
r\to\sqrt{\frac\lambda\pi}e^{-\lambda r^2}.
$
Consequently,
\begin{equation}\label{eq:Hn-zonoid-representation}
 H_n(x,t)
 =c\abs{x}
 +\int_{(0,\infty)}\int_{\R^{n-1}}
 \abs{\ip{x}{z}+t}\,
 \dd\gamma_{\lambda,n-1}(z)\,\dd\eta(\lambda),
\end{equation}
 where $\gamma_{\lambda,n-1}$ is  the  Gaussian probability measure
on $\R^{n-1}$ with density
$z\to
 \left(\frac\lambda\pi\right)^{(n-1)/2}e^{-\lambda\abs{z}^2}.
$
\end{lemma}

\begin{proof}
The Hausdorff--Bernstein--Widder theorem, in the form
recalled in \eqref{eq:HBW}, provides a finite positive measure $\nu$
on $[0,\infty)$ such that
\begin{equation}\label{eq:psi-second-Laplace}
 \psi''(s)=\int_{[0,\infty)}e^{-\lambda s^2}
 \,\dd\nu(\lambda),
 \qquad s\geq0.
\end{equation}
Since $\psi$ is even, $\psi'(0)=0$, and
\[
 \psi(s)-s\psi'(s)
 =\psi(0)-\int_0^s r\psi''(r)\,\dd r.
\]
Using \eqref{eq:mixture-intercept}, Tonelli's theorem, and
\eqref{eq:psi-second-Laplace}, one first obtains, as an identity of
extended nonnegative integrals,
\[
 \psi(0)-c
 =\int_0^\infty r\psi''(r)\,\dd r
 =\int_{[0,\infty)}
   \left(\int_0^\infty r e^{-\lambda r^2}\,\dd r\right)
   \dd\nu(\lambda).
\]
The left-hand side is finite, whereas the inner integral is
$+\infty$ when $\lambda=0$.  Hence $\nu(\{0\})=0$.  For
$\lambda>0$ the inner integral equals $(2\lambda)^{-1}$, and therefore
\begin{equation}\label{eq:nu-minus-one-moment}
 \psi(0)-c
 =\int_0^\infty r\psi''(r)\,\dd r
 =\int_{(0,\infty)}\frac1{2\lambda}\,\dd\nu(\lambda).
\end{equation}
Define
\begin{equation}\label{eq:eta-definition}
 \dd\eta(\lambda):=
 \frac{\sqrt\pi}{2\sqrt\lambda}\,\dd\nu(\lambda).
\end{equation}
The measure $\eta$ is finite.  On $(0,1)$ this follows from
$\lambda^{-1/2}\leq\lambda^{-1}$ and
\eqref{eq:nu-minus-one-moment}; on $[1,\infty)$ it follows from the
finiteness of $\nu$.

We next verify that the limiting slope is finite.  Since $\psi$ is
even, $\psi'(0)=0$, and therefore
\[
 \ell=\lim_{s\to\infty}\psi'(s)
 =\int_0^\infty\psi''(r)\,\dd r<\infty;
\]
indeed, $\psi''$ is continuous near zero and
$\psi''(r)\leq r\psi''(r)$ for $r\geq1$.
Moreover, Tonelli's theorem and
\eqref{eq:eta-definition} give
\begin{equation}\label{eq:eta-total-mass}
 \eta((0,\infty))
 =\int_{(0,\infty)}\frac{\sqrt\pi}{2\sqrt\lambda}
       \,\dd\nu(\lambda)
 =\int_0^\infty\psi''(r)\,\dd r
 =\ell.
\end{equation}
Put $F_\lambda(s)=\E\abs{\xi_\lambda+s}$.  Direct differentiation
gives
\begin{equation}\label{eq:F-lambda-data}
 F_\lambda''(s)=2\sqrt{\frac\lambda\pi}e^{-\lambda s^2},
 \qquad
 F_\lambda(0)=\frac1{\sqrt{\pi\lambda}},
 \qquad
 F_\lambda'(0)=0.
\end{equation}
Define, for $s\in\R$,
\[
 \widetilde\psi(s)
:=c+\int_{(0,\infty)}F_\lambda(s)\,\dd\eta(\lambda).
\]
This integral is finite.  Indeed, by the triangle inequality and
\eqref{eq:F-lambda-data},
\[
 F_\lambda(s)
 \leq \abs{s}+F_\lambda(0)
 =\abs{s}+\frac1{\sqrt{\pi\lambda}},
\]
while
\[
 \int_{(0,\infty)}\lambda^{-1/2}\,\dd\eta(\lambda)
 =\frac{\sqrt\pi}{2}
   \int_{(0,\infty)}\lambda^{-1}\,\dd\nu(\lambda)
 <\infty
\]
by \eqref{eq:nu-minus-one-moment}.  Moreover,
\[
 \abs{F_\lambda'(s)}\leq1,
 \qquad
 0\leq F_\lambda''(s)
 \leq2\sqrt{\frac{\lambda}{\pi}},
\]
and
\[
 \int_{(0,\infty)}\lambda^{1/2}\,\dd\eta(\lambda)
 =\frac{\sqrt\pi}{2}\nu((0,\infty))<\infty.
\]
Therefore, dominated convergence justifies differentiating twice
under the integral sign.

Using \eqref{eq:eta-definition} and
\eqref{eq:F-lambda-data}, we obtain
$$
 \widetilde\psi''(s)
 =\int_{(0,\infty)}
   2\sqrt{\frac{\lambda}{\pi}}e^{-\lambda s^2}
\,\dd\eta(\lambda)=\int_{(0,\infty)}e^{-\lambda s^2}\,\dd\nu(\lambda)
 =\psi''(s),
 $$
where the last equality follows from
\eqref{eq:psi-second-Laplace}.  Moreover,
$$
 \widetilde\psi(0)
 =c+\int_{(0,\infty)} \frac1{\sqrt{\pi\lambda}}\,\dd\eta(\lambda) =c+\int_{(0,\infty)}
   \frac1{2\lambda}\,\dd\nu(\lambda)
 =\psi(0)
$$
by \eqref{eq:nu-minus-one-moment}, while
\[
 \widetilde\psi'(0)
 =\int_{(0,\infty)}F_\lambda'(0)\,\dd\eta(\lambda)
 =0=\psi'(0).
\]
Thus $\widetilde\psi-\psi$ has zero second derivative and vanishing
value and derivative at zero.  Hence $\widetilde\psi=\psi$ on
$[0,\infty)$, and therefore on $\R$ by evenness.  This proves
\eqref{eq:psi-Gaussian-mixture}.

Rotational invariance gives, for $x\neq0$,
\[
 \int_{\R^{n-1}}\abs{\ip{x}{z}+t}\,
 \dd\gamma_{\lambda,n-1}(z)
 =\abs{x}F_\lambda(t/\abs{x}).
\]
Therefore \eqref{eq:Hn-zonoid-representation} holds and the integral in \eqref{eq:Hn-zonoid-representation} is finite: in addition to $\eta((0,\infty))<\infty$,
\[
 \int_{(0,\infty)}\lambda^{-1/2}\,\dd\eta(\lambda)
 =\frac{\sqrt\pi}{2}
 \int_{(0,\infty)}\lambda^{-1}\,\dd\nu(\lambda)<\infty
\]
by \eqref{eq:nu-minus-one-moment}.  After normalizing the
finite measure $\eta$ when it is nonzero, the integral term in
\eqref{eq:Hn-zonoid-representation} is a positive multiple of a
support function of the form \eqref{eq:random-zonoid}; if $\eta=0$,
the assertion is immediate.  The term $c\abs{x}$ is the support function
of an equatorial Euclidean ball and is also a zonoid. By \eqref{eq:eta-total-mass}, the same representation at
$x=0$ equals $\eta((0,\infty))\abs t=\ell\abs t$, as required.  If
$\ell>0$, the integral term is positive for every nonzero $(x,t)$:
for $x\neq0$ the Gaussian affine functional is not almost surely
zero, while for $x=0$ its value is $\ell\abs t$.  Hence $H_n$ is
positive away from the origin, and the represented zonoid has
nonempty interior.  This proves the claim.
\end{proof}

We now specialize the criterion to
$\psi=\varphi^\circ$.  Formula
\eqref{eq:polar-profile-definition} shows that $\varphi^\circ$ is even
and convex.   For smoothness, consider the function
\[
 a\longmapsto e^{a^2}\erf(a),\qquad a\in\R.
\]
It is odd and smooth, tends to $\pm\infty$ as $a\to\pm\infty$, and has
derivative
\[
 \frac2{\sqrt\pi}+2ae^{a^2}\erf(a)>0.
\]
It is therefore a smooth bijection of $\R$ onto itself with nowhere
vanishing derivative.  Its inverse is smooth and odd, and
\eqref{eq:polar-parametric} consequently defines $\varphi^\circ$ as a
smooth even function.  Finally, \eqref{eq:polar-intercept} and
\eqref{sloplimit} give $c=0$ and $\ell=1$,
respectively,  while \eqref{eq:polar-support-profile} identifies the
one-homogeneous function in \eqref{eq:Hn-profile} with
$h_{Z_n^\circ}$ for $x\neq0$; equality at $x=0$ follows by continuity
and $\ell=1$.  Thus the only remaining hypothesis is the complete
monotonicity of
\[
 u\longmapsto(\varphi^\circ)''(\sqrt u).
\]

\subsection{A complete-monotonicity
theorem involving the error function}
\label{sec:complete-monotonicity}

We now verify this remaining analytic hypothesis.  The
scaling in \eqref{eq:polar-parametric} is simplified by introducing
\begin{equation}\label{eq:G-g-definition}
 G(a):=\sqrt\pi e^{a^2}\erf(a),
 \qquad g(y):=e^{(G^{-1}(y))^2},
 \qquad a,y\geq0.
\end{equation}
Since $G(a)=\sqrt\pi s$ in \eqref{eq:polar-parametric},
\begin{equation}\label{eq:polar-g-scaling}
 \varphi^\circ(s)=g(\sqrt\pi s), \mbox{ for }  s\ge0.
\end{equation}
The central analytic statement is the following.

\begin{theorem}\label{thm:complete-monotonicity}
For $G$ and $g$ defined by \eqref{eq:G-g-definition}, the function
\begin{equation}\label{eq:Q-definition}
 Q(u):=g''(\sqrt u),\qquad u\geq0,
\end{equation}
is completely monotone.
\end{theorem}

The proof will be divided into two lemmas.  The first one identifies an
operator preserving complete monotonicity.

\begin{lemma}\label{lem:L-preserves-Bernstein}
Let $f:[0,\infty)\to[0,\infty)$ be completely monotone and
$f(0)<\infty$.  Define
\begin{equation}\label{eq:L-operator}
 (\mathcal Lf)(u):=\frac u2\int_0^1 f(ut^2)\,\dd t.
\end{equation}
Then $\mathcal Lf$ is a Bernstein function.  Consequently,
\begin{equation}\label{eq:S-operator}
 \mathcal Sf:=\frac1{1+\mathcal Lf}
\end{equation}
is completely monotone.
\end{lemma}

\begin{proof}
By \eqref{eq:HBW}, there is a finite positive measure $\mu$ such that
\[
 f(u)=\int_{[0,\infty)}e^{-ru}\,\dd\mu(r).
\]
Put
\[
 J_r(u):=\frac u2\int_0^1e^{-rut^2}\,\dd t.
\]
For $r\geq0$, differentiation under the integral sign together with integration by parts gives
\begin{equation}\label{eq:J-derivative}
 J_r'(u)
 =\frac12\int_0^1(1-rut^2)e^{-rut^2}\,\dd t
 =\frac14\int_0^1e^{-rut^2}\,\dd t+\frac14e^{-ru}.
\end{equation}
Therefore, for every integer $k\geq0$,
\[
 (-1)^k\frac{\mathrm d^k}{\mathrm d u^k}J_r'(u)
 =\frac14\int_0^1(rt^2)^k e^{-rut^2}\,\dd t
  +\frac14r^k e^{-ru}
 \geq0.
\]
Thus $J_r'$ is completely monotone for every $r\geq0$. Next Tonelli's theorem gives
\[
 \mathcal Lf(u)=\int_{[0,\infty)}J_r(u)\,\dd\mu(r),
\]
and \eqref{eq:J-derivative} shows that $(\mathcal Lf)'$ is completely
monotone.  Hence $\mathcal Lf$ is a Bernstein function.  Since
$x\mapsto(1+x)^{-1}$ is completely monotone, the composition theorem
for Bernstein functions, see  (ii) in Subsection \ref{compmonotone},  proves that $\mathcal Sf$ is completely
monotone. 
\end{proof}

\begin{lemma}\label{lem:P-completely-monotone}
Define
\begin{equation}\label{eq:P-definition}
 P(u):=\frac1{1+\sqrt u\,G^{-1}(\sqrt u)},
 \qquad u\geq0.
\end{equation}
Then $P$ is completely monotone.
\end{lemma}

\begin{proof}
Recall that $G(a)=\sqrt\pi\,e^{a^2}\erf(a)$ and
$
 \erf'(a)=\frac{2}{\sqrt\pi}e^{-a^2},
$ which gives
\begin{equation}\label{difeqq}
G'(a)
 =\sqrt\pi\left(
     2ae^{a^2}\erf(a)+e^{a^2}\erf'(a)
   \right)
 =2a\sqrt\pi\,e^{a^2}\erf(a)+2
 =2\bigl(1+aG(a)\bigr)>0.
\end{equation}
Thus $G$ maps $[0,\infty)$ bijectively onto $[0,\infty)$ and
\begin{equation}\label{eq:G-inverse-ODE}
 (G^{-1})'(y)=\frac1{2(1+yG^{-1}(y))}.
\end{equation}
By \eqref{eq:P-definition} and $G^{-1}(0)=0$,
\[
 G^{-1}(y)=\frac12\int_0^yP(v^2)\,\dd v.
\]
Substitution $v=yt$, $u=y^2$, yields the exact fixed-point identity
\begin{equation}\label{eq:P-fixed-point}
 P(u)=\frac{1}{1+\frac u2\int_0^1P(ut^2)\,\dd t}=\frac{1}{1+\mathcal L P(u)}
 =\mathcal SP(u),
\end{equation}
where $\mathcal S$ is an operator defined in \eqref{eq:S-operator}.
Let $P_0=1$ and set $P_{k+1}=\mathcal SP_k$.  By
Lemma~\ref{lem:L-preserves-Bernstein}, every $P_k$ is completely
monotone.  Since $0<P\leq1$ and
\[
 \abs{\frac1{1+a}-\frac1{1+b}}\leq\abs{a-b}
 \qquad \mbox{ for }a,b\geq0,
\]
the fixed-point identity \eqref{eq:P-fixed-point} implies
\[
 \abs{P_{k+1}-P}\leq\mathcal L\abs{P_k-P}.
\]
Inductively,
\begin{equation}\label{eq:P-iteration-error}
 \abs{P_k(u)-P(u)}\leq {\mathcal L}^k \abs{P_0-P} (u)\leq(\mathcal L^k1)(u),
\end{equation}
where we used $P_0=1$ and $0\le P  \le 1$. Next, for every integer $j\geq0$,
\[
 \mathcal L(x^j)(u)=\frac{u^{j+1}}{2(2j+1)},
\]
and therefore
\begin{equation}\label{eq:Lk-one}
 (\mathcal L^k1)(u)=\frac{k!}{(2k)!}u^k.
\end{equation}
The last expression tends to zero uniformly on compact intervals.
  Consequently, $P_k\to P$ locally uniformly and hence pointwise.  Since
each $P_k$ is completely monotone, property~(i) in
Subsection~\ref{compmonotone} implies that $P$ is completely monotone. 
\end{proof}

\begin{proof}[Proof of Theorem~\ref{thm:complete-monotonicity}]
We begin by computing $Q(u)=g''(\sqrt{u})$.  Since
\[
 g(y)=e^{(G^{-1}(y))^2},
\]
equation \eqref{eq:G-inverse-ODE} gives
\[
\begin{aligned}
 g'(y)
 &=2(G^{-1})'(y)G^{-1}(y)e^{(G^{-1}(y))^2} \\
 &=\frac{G^{-1}(y)e^{(G^{-1}(y))^2}}
         {1+yG^{-1}(y)}.
\end{aligned}
\]
Differentiating once more and again using
\eqref{eq:G-inverse-ODE}, we obtain
\begin{align*}
 g''(y)
 &=
 \frac{
 e^{(G^{-1}(y))^2}
 \bigl(1+2(G^{-1}(y))^2\bigr)(G^{-1})'(y)}
 {1+yG^{-1}(y)}-
 \frac{
 G^{-1}(y)e^{(G^{-1}(y))^2}
 \bigl(G^{-1}(y)+y(G^{-1})'(y)\bigr)}
 {\bigl(1+yG^{-1}(y)\bigr)^2}
 \\
 &=
 \frac{e^{(G^{-1}(y))^2}}
      {\bigl(1+yG^{-1}(y)\bigr)^2}
 \left[
 (G^{-1})'(y)
 \Bigl(
  1+2(G^{-1}(y))^2+2y(G^{-1}(y))^3
 \Bigr)
 -(G^{-1}(y))^2
 \right]
 \\
 &=
 \frac{e^{(G^{-1}(y))^2}}
      {\bigl(1+yG^{-1}(y)\bigr)^2}
 \left[
 \frac{
  1+2(G^{-1}(y))^2+2y(G^{-1}(y))^3}
 {2\bigl(1+yG^{-1}(y)\bigr)}
 -(G^{-1}(y))^2
 \right]
 \\
 &=
 \frac{e^{(G^{-1}(y))^2}}
      {2\bigl(1+yG^{-1}(y)\bigr)^3}.
\end{align*}
By \eqref{eq:P-definition}, for $y\geq0$,
\[
 \frac1{1+yG^{-1}(y)}=P(y^2).
\]
Consequently,
\begin{equation}\label{eq:Q-P-formula}
 Q(u)=g''(\sqrt{u})
 =\frac12e^{(G^{-1}(\sqrt{u}))^2}P(u)^3.
\end{equation}

Although $P^3$ is completely monotone,
\eqref{eq:Q-P-formula} does not immediately imply that $Q$ is
completely monotone, because the exponential factor is increasing.
Since $P(0)=1$ and $G^{-1}(0)=0$, formula
\eqref{eq:Q-P-formula} gives $Q(0)=1/2$.  Moreover, $Q(u)>0$.
Define
\begin{equation}\label{eq:B}
 B(u):=-\log\frac{Q(u)}{Q(0)}.
\end{equation}
Then $B(0)=0$ and
\[
 B'(u)=-\frac{Q'(u)}{Q(u)}.
\]
It is therefore enough to prove that $B'$ is completely monotone.
Indeed, this would imply that $B'\geq0$ and hence that $B\geq0$, so
$B$ would be a Bernstein function.  Property~(iii) in
Subsection~\ref{compmonotone} would then give
\[
 Q(u)=Q(0)e^{-B(u)}=\frac12e^{-B(u)},
\]
which is completely monotone.  We now express $B'$ in terms of $P$.  By
\eqref{eq:G-inverse-ODE} and \eqref{eq:P-definition}, for $u>0$,
\begin{equation}\label{eq:P-inverse-identities}
 (G^{-1})'(\sqrt{u})=\frac12P(u),
 \qquad
 G^{-1}(\sqrt{u})
 =\frac{1-P(u)}{\sqrt{u}\,P(u)}.
\end{equation}
Differentiating \eqref{eq:P-definition} directly and using
\eqref{eq:P-inverse-identities}, we obtain
\begin{align}
 P'(u)
 &=-P(u)^2\left(
 \frac{G^{-1}(\sqrt{u})}{2\sqrt{u}}
 +\frac12(G^{-1})'(\sqrt{u})
 \right) \notag\\
 &=-P(u)^2\left(
 \frac{1-P(u)}{2uP(u)}+\frac14P(u)
 \right) \notag\\
 &=-\frac{P(u)(1-P(u))}{2u}-\frac14P(u)^3.
 \label{eq:P-derivative}
\end{align}
The same identities give
\begin{equation}\label{eq:G-inverse-square-derivative}
 \frac{\dd}{\dd u}\bigl(G^{-1}(\sqrt{u})\bigr)^2
 =
 \frac{G^{-1}(\sqrt{u})}{\sqrt{u}}\,
 (G^{-1})'(\sqrt{u})
 =\frac{1-P(u)}{2u}.
\end{equation}

We now differentiate \eqref{eq:Q-P-formula} directly.  Using
\eqref{eq:P-derivative} and
\eqref{eq:G-inverse-square-derivative}, we find
\begin{align*}
 Q'(u)
 &=
 \frac12e^{(G^{-1}(\sqrt{u}))^2}P(u)^3
 \frac{1-P(u)}{2u}
 +\frac32e^{(G^{-1}(\sqrt{u}))^2}P(u)^2P'(u)
 \\
 &=
 \frac12e^{(G^{-1}(\sqrt{u}))^2}P(u)^3
 \frac{1-P(u)}{2u}
 \\
 &\qquad
 +\frac32e^{(G^{-1}(\sqrt{u}))^2}P(u)^2
 \left(
 -\frac{P(u)(1-P(u))}{2u}-\frac14P(u)^3
 \right)
 \\
 &=
 -\frac12e^{(G^{-1}(\sqrt{u}))^2}P(u)^3
 \left(
 \frac{1-P(u)}u+\frac34P(u)^2
 \right)
 \\
 &=
 -Q(u)\left(
 \frac{1-P(u)}u+\frac34P(u)^2
 \right).
\end{align*}
Since $Q(u)>0$, it follows that
\begin{equation}\label{eq:Q-log-derivative}
 B'(u)=-\frac{Q'(u)}{Q(u)}
 =\frac{1-P(u)}u+\frac34P(u)^2.
\end{equation}

The product $P^2$ is completely monotone by property~(i) in
Subsection~\ref{compmonotone}.  To treat the other term, the
Hausdorff--Bernstein--Widder theorem gives a positive measure $\rho$
on $[0,\infty)$ such that
\[
 P(u)=\int_{[0,\infty)}e^{-ur}\,\dd\rho(r).
\]
Since
$
 \rho([0,\infty))=P(0)=1,
$
for $u>0$ we have
$$
 \frac{1-P(u)}u
 =
 \int_{[0,\infty)}\frac{1-e^{-ur}}u\,\dd\rho(r) =
 \int_{[0,\infty)}\int_0^r e^{-us}\,\dd s\,\dd\rho(r)=
 \int_0^\infty e^{-us}\rho([s,\infty))\,\dd s,
$$
where the last equality follows from Tonelli's theorem.  Thus
$(1-P(u))/u$ is the Laplace transform of a positive measure and is
therefore completely monotone.

Thus $B'(u)$ is completely monotone.  Moreover, $B(0)=0$ and $B'(u)\geq0$, so
$B(u)\geq0$.  Hence $B$ is a Bernstein function.  Finally,
\eqref{eq:B} gives
\[
 Q(u)=Q(0)e^{-B(u)}=\frac12e^{-B(u)},
\]
which is completely monotone by property~(iii) in
Subsection~\ref{compmonotone}.  This completes the proof.
\end{proof}

\begin{corollary}\label{cor:polar-profile-CM}
The function
\begin{equation}\label{eq:polar-profile-CM}
 u\longmapsto(\varphi^\circ)''(\sqrt u)
\end{equation}
is completely monotone on $[0,\infty)$.
\end{corollary}

\begin{proof}
By \eqref{eq:polar-g-scaling},
\[
 (\varphi^\circ)''(\sqrt u)
 =\pi g''(\sqrt{\pi u})=\pi Q(\pi u).
\]
Positive multiplication and positive dilation preserve complete
monotonicity.  The conclusion follows from
Theorem~\ref{thm:complete-monotonicity}.
\end{proof}

\begin{remark}\label{rem:MSE-CM}
The specific complete-monotonicity question resolved by Theorem~\ref{thm:complete-monotonicity} was posed on Mathematics Stack Exchange in March 2022 by the user \texttt{uniquesolution} \cite{MSE2022}, but no complete solution was given there.
 
\end{remark}

We now combine Corollary~\ref{cor:polar-profile-CM} with
the criterion of Subsection~\ref{subsec:zonoid-criterion}.

\begin{theorem}\label{thm:polar-is-zonoid}
For every $n\geq2$, the polar $Z_n^\circ$ is a zonoid.
More precisely, if
\begin{equation}\label{eq:polar-Bernstein-measure}
 (\varphi^\circ)''(\sqrt u)
 =\int_{(0,\infty)}e^{-\lambda u}\,\dd\nu(\lambda),
\end{equation}
then $h_{Z_n^\circ}$ has the  Gaussian-mixture representation
\begin{equation}\label{eq:polar-explicit-mixture}
 h_{Z_n^\circ}(x,t)
 =\int_{(0,\infty)}\int_{\R^{n-1}}
 \abs{\ip{x}{z}+t}\,
 \dd\gamma_{\lambda,n-1}(z)
 \frac{\sqrt\pi}{2\sqrt\lambda}\,\dd\nu(\lambda).
\end{equation}
Here $\gamma_{\lambda,n-1}$ denotes the centered Gaussian
probability measure on $\R^{n-1}$ whose density at $z$ is
$({\lambda}/{\pi})^{(n-1)/2}e^{-\lambda\abs{z}^2}$.
\end{theorem}

\begin{proof}
Corollary~\ref{cor:polar-profile-CM} provides
complete monotonicity, and the Hausdorff--Bernstein--Widder theorem
gives a finite representing measure $\nu$ on $[0,\infty)$.  The
intercept argument leading to \eqref{eq:nu-minus-one-moment} shows that
$\nu(\{0\})=0$, giving \eqref{eq:polar-Bernstein-measure}.  All other
hypotheses of Lemma~\ref{lem:Gaussian-mixture} were verified at the end
of Subsection~\ref{subsec:zonoid-criterion}.  Applying the lemma and
using \eqref{eq:eta-definition} gives
\eqref{eq:polar-explicit-mixture}.
\end{proof}

Combining the  zonoid representation
\eqref{eq:Zn-intro} with Theorem~\ref{thm:polar-is-zonoid} proves the
first assertion of Theorem~\ref{thm:main-finite}.  Notice that the same
Hausdorff--Bernstein--Widder representing measure works in every dimension.

\section{The exact Banach--Mazur distance}\label{sec:exact-BM}

Mathis proved sharp concentric ball inclusions for $Z_n$
\cite{Mathis2025}, which give an upper bound for its
Banach--Mazur distance from the Euclidean ball. We show that
no other comparison ellipsoid improves this bound, and hence
determine the exact Banach--Mazur distance.

\subsection{Reduction to a one-variable problem}
We begin by exploiting the orthogonal symmetries of $Z_n$. The following standard averaging lemma shows that any ellipsoidal sandwich may be replaced, without increasing its factor, by one whose ellipsoid is invariant under the same compact symmetry group. Applied to $Z_n$, this reduces the comparison ellipsoids to a one-parameter family.

\begin{lemma}
\label{lem:ellipsoid-averaging}
Let $K\subset\R^n$ be a convex body invariant under a compact
subgroup $\mathcal G\subset O(n)$.  If
\[
 E\subset K\subset aE
\]
for an origin-centered ellipsoid $E$, then there is a
$\mathcal G$-invariant origin-centered ellipsoid $\overline E$ such
that
\[
 \overline E\subset K\subset a\overline E.
\]
\end{lemma}

\begin{proof}
Note that  $h_E(u)^2$ is a positive-definite quadratic
form.  The sandwich is equivalent to
\begin{equation}\label{eq:quadratic-sandwich}
 \frac{h_K(u)^2}{a^2}\leq h_E(u)^2\leq h_K(u)^2
 \qquad \mbox{ for } u\in\R^n.
\end{equation}
Let $\mu$ be Haar probability measure on $\mathcal G$ and set
\[
 \overline h^2(u)=\int_{\mathcal G}h_E^2(gu)\,\dd\mu(g).
\]
Because $h_K(gu)=h_K(u)$, averaging
\eqref{eq:quadratic-sandwich} preserves both inequalities.
The form $\overline h^2$ is positive definite and
$\mathcal G$-invariant, and hence $\overline h$ is the support function
of an invariant ellipsoid $\overline E$ with the same sandwich factor.
\end{proof}
For $0\leq s<\infty$, set
\[
 u_s:=\frac{(e_1,s)}{\sqrt{1+s^2}}
      \in \Sph^{n-1}\cap\operatorname{span}\{e_1,e_n\}.
\]
Thus $s$ parametrizes the Euclidean unit directions in a meridian of
$Z_n$, with $u_0=e_1$ and $u_s\to e_n$ as $s\to\infty$.  By the
one-homogeneity of the support function and the definition
$\varphi(s)=h_{Z_n}(e_1,s)$,
\[
 h_{Z_n}(u_s)
 =\frac{h_{Z_n}(e_1,s)}{\sqrt{1+s^2}}
 =\frac{\varphi(s)}{\sqrt{1+s^2}}.
\]
Let
\begin{equation}\label{eq:F1-definition}
 F_1(s):=\frac{\varphi(s)}{\sqrt{1+s^2}},
 \qquad 0\leq s<\infty,
\end{equation}
and $F_1(\infty)=1,$ by continuity. Consequently, $F_1$ is the restriction of $h_{Z_n}$ to the Euclidean
unit directions in a meridian, parametrized by their slope.  Since
$Z_n$ is a body of revolution and is origin-symmetric, these
directions represent, up to symmetry, all directions in $\Sph^{n-1}$.
Its endpoint values are 
\begin{equation}\label{eq:F1-endpoints}
 F_1(0)=F_1(\infty)=1.
\end{equation}
Using \eqref{eq:phi-identities},
\begin{equation}\label{eq:F1-derivative}
 F_1'(s)=
 \frac{\erf(s/\sqrt\pi)-se^{-s^2/\pi}}
 {(1+s^2)^{3/2}}.
\end{equation}

\begin{lemma}\label{lem:b-infty-minimum}
There is a unique $s_0>0$ such that
\begin{equation}\label{eq:s0-equation}
 \erf\!\left(\frac{s_0}{\sqrt\pi}\right)
 =s_0e^{-s_0^2/\pi}.
\end{equation}
The function $F_1$ decreases on $(0,s_0)$ and increases on
$(s_0,\infty)$.  Consequently,
\begin{equation}\label{eq:F1-range}
 b_\infty\leq F_1(s)\leq1,
 \qquad \mbox{ where }
 b_\infty=F_1(s_0)
 =e^{-s_0^2/\pi}\sqrt{1+s_0^2}.
\end{equation}
In particular,
\[
 0<b_\infty<1.
\]
\end{lemma}

\begin{proof}
Set
\[
 H(s):=\erf\!\left(\frac{s}{\sqrt\pi}\right)
       -se^{-s^2/\pi},
 \qquad s\geq0.
\]
Differentiating gives
\[
 H'(s)
 =\frac2\pi e^{-s^2/\pi}
  -\left(1-\frac{2s^2}{\pi}\right)e^{-s^2/\pi}
 =\frac{2s^2+2-\pi}{\pi}e^{-s^2/\pi}.
\]
Thus $H$ strictly decreases on
$(0,\sqrt{(\pi-2)/2})$ and strictly increases on
$(\sqrt{(\pi-2)/2},\infty)$.
Since $H(0)=0$, its minimum is strictly negative.
Moreover,
$
 \lim_{s\to\infty}H(s)=1,
$
because $\erf(s/\sqrt\pi)\to1$ and
$se^{-s^2/\pi}\to0$.
It follows that $H$ has exactly one positive zero $s_0$,
with
\[
 H(s)<0 \mbox{ when } 0<s<s_0,
 \qquad \mbox{ and }
 H(s)>0 \mbox{ when } s>s_0.
\]
The equation $H(s_0)=0$ is precisely
\eqref{eq:s0-equation}.

By \eqref{eq:F1-derivative},
\[
 F_1'(s)=\frac{H(s)}{(1+s^2)^{3/2}}.
\]
Consequently, $F_1$ strictly decreases on $(0,s_0)$ and
strictly increases on $(s_0,\infty)$. Together with
$F_1(0)=1$ and $\lim_{s\to\infty}F_1(s)=1$, this gives
\[
 b_\infty=F_1(s_0)\leq F_1(s)\leq1,
 \qquad s\geq0.
\]

Finally, \eqref{eq:s0-equation} and
\eqref{eq:phi-definition} give
\[
 \varphi(s_0)
 =e^{-s_0^2/\pi}
  +s_0\erf\!\left(\frac{s_0}{\sqrt\pi}\right)
 =(1+s_0^2)e^{-s_0^2/\pi}.
\]
Hence
\[
 b_\infty
 =\frac{\varphi(s_0)}{\sqrt{1+s_0^2}}
 =e^{-s_0^2/\pi}\sqrt{1+s_0^2}>0.
\]
Since $F_1$ strictly decreases from $F_1(0)=1$ to
$F_1(s_0)=b_\infty$, we also have $b_\infty<1$.
\end{proof}

\subsection{The optimal Banach--Mazur ellipsoid}

We now combine the preceding symmetry reduction with the analysis of
the meridional support ratio to optimize over all comparison ellipsoids.

\begin{theorem}\label{thm:exact-BM}
For every $n\geq2$,
\begin{equation}\label{eq:exact-BM}
 d_{\BM}(Z_n,B_2^n)=b_\infty^{-1}.
\end{equation}
The Euclidean ball is, up to scaling, an optimal Banach--Mazur ellipsoid.
\end{theorem}

\begin{proof}
We first explain why this calculation accounts for every
linear position in the definition of the Banach--Mazur distance.  The
sandwich
\[
 B_2^n\subset TZ_n\subset aB_2^n,
 \qquad T\in GL(n),
\]
is equivalent, after applying $T^{-1}$, to
\[
 T^{-1}B_2^n\subset Z_n\subset aT^{-1}B_2^n.
\]
Conversely, every origin-centered ellipsoid is $T^{-1}B_2^n$ for some
$T\in GL(n)$.  Thus the ellipsoid formulation
\eqref{eq:BM-ellipsoids} loses no possible linear position.

The symmetry group of $Z_n$ contains
\[
 \mathcal G=O(n-1)\times\{\pm1\},
\]
acting by $(R,\varepsilon)(x,t)=(Rx,\varepsilon t)$.  If an arbitrary
ellipsoid $E$ gives a sandwich
$E\subset Z_n\subset aE$, Lemma~\ref{lem:ellipsoid-averaging}
produces a $\mathcal G$-invariant ellipsoid with the same factor $a$.
Consequently, no nonsymmetric ellipsoid can improve the infimum over
all $\mathcal G$-invariant ellipsoids.

Let $E$ be such an invariant ellipsoid.  The square of its support
function has the form
\[
 h_E(x,t)^2
 =\ip{Cx}{x}+2t\ip{v}{x}+Bt^2,
\]
where $C$ is a positive-definite symmetric $(n-1)\times(n-1)$ matrix,
$v\in\R^{n-1}$, and $B>0$.  Reflection in the last coordinate forces
$v=0$, and invariance under every $R\in O(n-1)$ forces
$C=AI_{n-1}$ for some $A>0$.  Hence
\[
 h_E(x,t)^2=A\abs{x}^2+Bt^2,
 \qquad A,B>0.
\]
Writing $\alpha=B/A$, we see that $E=\sqrt A\,E_\alpha$, where
\begin{equation}\label{eq:Ealpha}
 h_{E_\alpha}(x,t)^2=\abs{x}^2+\alpha t^2,
 \qquad \alpha>0.
\end{equation}
Equivalently,
\[
 E_\alpha
 =\left\{(y,z)\in\R^{n-1}\times\R:
          \abs{y}^2+\frac{z^2}{\alpha}\leq1\right\}.
\]
Thus $E_\alpha$ has horizontal semiaxes $1$ and axial semiaxis
$\sqrt\alpha$.  The parameter $\alpha$ determines its shape, whereas
$\sqrt A$ is only an overall homothety.

We next compute the optimal homothety for a fixed shape.  For
$x\neq0$, put
\[
 s=\frac{\abs{t}}{\abs{x}}.
\]
Using homogeneity and evenness, the quotient of support functions is
\begin{equation}\label{eq:Falpha}
 \frac{h_{Z_n}(x,t)}{h_{E_\alpha}(x,t)}
 =F_\alpha(s)
 :=\frac{\varphi(s)}{\sqrt{1+\alpha s^2}}.
\end{equation}
The axial direction $x=0$ corresponds to $s=\infty$.  Since
$\varphi(s)/s\to1$, $F_\alpha$ extends continuously to the
compactified interval $[0,\infty]$, with
\[
 F_\alpha(0)=1,
 \qquad
 F_\alpha(\infty)=\alpha^{-1/2}.
\]
Every nonzero direction is represented by one of these values of
$s$, because the quotient is homogeneous of degree zero and invariant
under $\mathcal G$.  Put
\[
 m_\alpha:=\min_{0\leq s\leq\infty}F_\alpha(s),
 \qquad
 M_\alpha:=\max_{0\leq s\leq\infty}F_\alpha(s).
\]
Then
\[
 m_\alpha h_{E_\alpha}(u)
 \leq h_{Z_n}(u)
 \leq M_\alpha h_{E_\alpha}(u)
 \qquad\mbox{ for all } u\in\R^n.
\]
Comparison of support functions is equivalent to inclusion, and
therefore
\[
 m_\alpha E_\alpha
 \subset Z_n
 \subset M_\alpha E_\alpha
 =\frac{M_\alpha}{m_\alpha}
   \bigl(m_\alpha E_\alpha\bigr).
\]
Conversely, suppose that a homothetic copy $\rho E_\alpha$ satisfies
\[
 \rho E_\alpha\subset Z_n\subset a\rho E_\alpha.
\]
The corresponding support inequalities say
\[
 \rho\leq F_\alpha(s)\leq a\rho, \qquad
\mbox{ for all } 0\leq s\leq\infty.
\]
Thus $\rho\leq m_\alpha$ and $M_\alpha\leq a\rho$, which imply
$a\geq M_\alpha/m_\alpha$.  The preceding sandwich attains this value
when $\rho=m_\alpha$.  Hence
\begin{equation}\label{eq:Delta-alpha}
 \begin{split}
 \Delta(E_\alpha,Z_n)
 &:=%
 \inf_{\rho>0}\inf\left\{
 a\geq1:\rho E_\alpha\subset Z_n\subset a\rho E_\alpha
 \right\}\\
 &=\frac{M_\alpha}{m_\alpha}
 =\frac{\displaystyle
        \max_{0\leq s\leq\infty}F_\alpha(s)}
       {\displaystyle
        \min_{0\leq s\leq\infty}F_\alpha(s)}.
 \end{split}
\end{equation}
This also explains explicitly why the scale $\rho$ need not be kept
as a second parameter: replacing $E_\alpha$ by $\rho E_\alpha$
divides both the maximum and minimum support ratios by $\rho$, so it
cancels from their quotient.  The parameter $\alpha$, on the other hand, determines the relative
scaling in the axial direction: the corresponding semiaxis is
$\sqrt{\alpha}$.  The symmetry
reduction now gives
\[
 d_{\BM}(Z_n,B_2^n)
 =\inf_{\alpha>0}\Delta(E_\alpha,Z_n).
\]

It remains to minimize this one-variable expression.  Recall that
$s_0$ is the unique minimizer of $F_1$ and
\[
 \varphi(s_0)=b_\infty\sqrt{1+s_0^2}.
\]
If $\alpha\geq1$, the horizontal direction gives
$M_\alpha\geq F_\alpha(0)=1$, whereas
\[
 m_\alpha
 \leq F_\alpha(s_0)
 =b_\infty
   \sqrt{\frac{1+s_0^2}{1+\alpha s_0^2}}
 \leq b_\infty.
\]
Thus $\Delta(E_\alpha,Z_n)\geq b_\infty^{-1}$.

If $0<\alpha\leq1$, the axial direction gives
$M_\alpha\geq F_\alpha(\infty)=\alpha^{-1/2}$, while
\[
 m_\alpha\leq b_\infty
 \sqrt{\frac{1+s_0^2}{1+\alpha s_0^2}}.
\]
It follows that
\begin{equation*}
 \Delta(E_\alpha,Z_n)
 \geq\frac{\alpha^{-1/2}}
{b_\infty\sqrt{(1+s_0^2)/(1+\alpha s_0^2)}}=\frac1{b_\infty}
 \sqrt{\frac{1+\alpha s_0^2}{\alpha(1+s_0^2)}}
 \geq\frac1{b_\infty},
\end{equation*}
where the last inequality is equivalent to $\alpha\leq1$.

Finally, for $\alpha=1$, Lemma~\ref{lem:b-infty-minimum} gives
$m_1=b_\infty$ and $M_1=1$. Since $E_1=B_2^n$ we get $\Delta(E_1,Z_n)=b_\infty^{-1}$. 
Together with the preceding infimum over $\alpha$, this proves
\eqref{eq:exact-BM} and the optimality of the Euclidean ball.  The
sharp representative of its homothety class is
$b_\infty B_2^n$, giving
$
 b_\infty B_2^n\subset Z_n\subset B_2^n.
$
\end{proof}


To indicate the size of the constant in
\eqref{eq:exact-BM}, we record the following numerical
approximations.  They are included only for orientation and play no
role in the proof.  Solving the one-variable equation
\eqref{eq:s0-equation} gives, to two decimal places,
$$
 s_0\approx1.43, \qquad
 b_\infty\approx0.91 \qquad\mbox{and} \qquad
 b_\infty^{-1}\approx1.10.
$$
More concretely, if $
 H(s)=\erf(s/\sqrt\pi)-se^{-s^2/\pi},
$  as in the proof of Lemma \ref{lem:b-infty-minimum},
then $H(1.42)<0<H(1.43)$, so $1.42<s_0<1.43$.  These signs can be
seen numerically from
\[
H(1.42)\approx-0.0046<0<0.00027\approx H(1.43).
\]
 We note that the strict inequality $b_\infty<1$ is proved in
Lemma~\ref{lem:b-infty-minimum} without using these numerical values.

The inequality $b_\infty^{-1}>1.09$ can be checked directly,
without relying on the  decimal approximations.  Put
$x=s^2/\pi$.  Combining the power series for the exponential and the
error function gives
\[
 \varphi(s)=1+\sum_{j=1}^\infty
 (-1)^{j-1}\frac{x^j}{j!(2j-1)}.
\]
For $0<x<1$ the terms decrease in absolute value.  Hence, with
$p(x)=1+x-x^2/6+x^3/30$, the alternating-series estimate, the
elementary bound $\pi>31/10$, and $\sqrt{74}>43/5$ give
\[
 F_1\!\left(\frac75\right)
=\frac5{\sqrt{74}}\,\varphi\!\left(\frac75\right)
 <\frac{25}{43}\,
 p\!\left(\frac{98}{155}\right)
 <\frac{100}{109}.
\]
The polynomial  $p$ is increasing on $[0,1]$.  Since
$b_\infty\leq F_1(7/5)$, it follows that
\[
 b_\infty^{-1}>\frac{109}{100}=1.09.
\]

\section{An infinite-dimensional consequence}
\label{sec:infinite-dimensional}

This section explains why the dimension-free nature of the profile is
substantive.  All Banach and Hilbert spaces below are real.

Let $H$ be a separable infinite-dimensional Hilbert space.  On a standard 
atomless probability space 
$(\Omega,\mathcal F,\mathbb P)$ choose a
\emph{scaled isonormal Gaussian process} $(W_x)_{x\in H}$ satisfying
\begin{equation}\label{eq:isonormal-covariance}
 \E W_xW_y=\frac\pi2\ip{x}{y}_H.
\end{equation}
For example, if $(e_j)$ is an orthonormal basis of $H$ and
$(\gamma_j)$ are independent $N(0,\pi/2)$ variables, set
\[
 W_x=\sum_{j=1}^\infty\ip{x}{e_j}\gamma_j,
\]
with convergence in $L_2(\Omega)$.   Thus
$(W_x)_{x\in H}$ is a scalar-valued Gaussian process indexed by $H$;
the series is interpreted separately for each $x\in H$, rather than as
the inner product with an $H$-valued Gaussian random vector.  The
scalar process satisfies
\begin{equation}\label{eq:W-absolute-moment}
 \E\abs{W_x}=\norm{x}_H.
\end{equation}
On $H\oplus\R$ define
\begin{equation}\label{eq:X-norm}
\norm{(x,t)}_X:=\E\abs{W_x+t}.
\end{equation}
The map
\begin{equation}\label{eq:J-embedding}
 J:H\oplus\R\longrightarrow L_1(\Omega),
 \qquad J(x,t)=W_x+t,
\end{equation}
is linear and isometric.  It is injective because a Gaussian random
variable is almost surely constant only when its variance is zero.
Furthermore,
\begin{equation}\label{eq:X-Hilbert-equivalence}
 \max\{\norm{x}_H,\abs t\}
 \leq\norm{(x,t)}_X
 \leq\norm{x}_H+\abs t.
\end{equation}
The upper estimate is the triangle inequality.  The two
lower estimates, by $\abs t$ and $\norm{x}_H$, follow, respectively,
from Jensen's inequality and the symmetry of $W_x$:
\begin{align*}
 \E\abs{W_x+t}
 &\geq \abs{\E(W_x+t)}=\abs t,\\
 \E\abs{W_x+t}
 &=\frac12\E\bigl(\abs{W_x+t}+\abs{W_x-t}\bigr)
 \geq\E\abs{W_x}=\norm{x}_H.
\end{align*}
In the second line, the equality uses the fact that $W_x$ and $-W_x$
have the same distribution, while the inequality follows from
$\abs{a+b}+\abs{a-b}\geq2\abs a$.

Consequently,
\[
 \frac1{\sqrt2}
 \bigl(\norm{x}_H^2+t^2\bigr)^{1/2}
 \leq \norm{(x,t)}_X
 \leq
 \sqrt2\bigl(\norm{x}_H^2+t^2\bigr)^{1/2}.
\]
Thus $\norm{\cdot}_X$ is equivalent to the Hilbert norm on
$H\oplus_2\R$.  Since $H\oplus_2\R$ is complete, equivalence of the
two norms implies that $X$ is complete.  Moreover, $X$ is isomorphic
to the Hilbert space $H\oplus_2\R$ and is therefore reflexive.

For $r=\norm{x}_H>0$, \eqref{eq:folded-normal} gives
\begin{equation}\label{eq:X-profile}
 \norm{(x,t)}_X=r\varphi(t/r).
\end{equation}
Identify $(H\oplus\R)^*$ with $H\oplus\R$ through the underlying
Hilbert pairing.  Maximizing first over the direction of the horizontal
coordinate and then using homogeneity gives, for $y\ne0$
\begin{equation}\label{eq:X-dual-profile}
 \norm{(y,s)}_{X^*}
 =\sup_{(x,t)\neq0}
 \frac{\ip{x}{y}+ts}{\norm{(x,t)}_X}
 =\norm{y}_H\sup_{q\in\R}
 \frac{1+(s/\norm{y}_H)q}{\varphi(q)}
=\norm{y}_H\,\varphi^\circ(s/\norm{y}_H).
\end{equation}
The axial case follows by continuity.  If
$E\subset H$ is an $(n-1)$-dimensional subspace and
$X_E=(E\oplus\R,\norm{\cdot}_X)$, then, after an orthogonal
identification of $E$ with $\R^{n-1}$, the finite-dimensional
identities are
\[
 B_{X_E}=Z_n^\circ,
 \qquad B_{X_E^*}=Z_n.
\]

We next construct an explicit $L_1$ representation of the dual norm.
Let $\nu$ be the Hausdorff--Bernstein--Widder representing measure in
\eqref{eq:polar-Bernstein-measure}, and define
\begin{equation}\label{eq:eta-infinite}
 \dd\eta(\lambda)
 =\frac{\sqrt\pi}{2\sqrt\lambda}\,\dd\nu(\lambda).
\end{equation}
Lemma~\ref{lem:Gaussian-mixture}, together with
\eqref{eq:polar-intercept}, gives
\begin{equation}\label{eq:dual-profile-mixture-infinite}
 \varphi^\circ(a)
 =\int_{(0,\infty)}
 \E\abs{\frac{\gamma}{\sqrt{\pi\lambda}}+a}
 \,\dd\eta(\lambda),
\end{equation}
where $\gamma\sim N(0,\pi/2)$.  The normalization identities are
\begin{equation}\label{eq:eta-normalizations}
 \eta((0,\infty))=1,
 \qquad
 \int_{(0,\infty)}\frac1{\sqrt{\pi\lambda}}
 \,\dd\eta(\lambda)=1.
\end{equation}
The first identity follows from the limiting slope of
$\varphi^\circ$. Indeed, each Gaussian profile
\[
 F_\lambda(a):=
 \E\abs{\frac{\gamma}{\sqrt{\pi\lambda}}+a}
\]
satisfies
\[
 F_\lambda'(a)
 =\E\operatorname{sgn}\left(
 a+\frac{\gamma}{\sqrt{\pi\lambda}}\right)
 \longrightarrow 1
 \qquad \mbox{as } a\to\infty,
\]
and $\abs{F_\lambda'(a)}\leq1$.  Since $\eta$ is finite,
dominated convergence in
\eqref{eq:dual-profile-mixture-infinite} gives
\[
 \lim_{a\to\infty}(\varphi^\circ)'(a)
 =\eta((0,\infty)).
\]
Thus \eqref{sloplimit}  implies
$\eta((0,\infty))=1$. The second identity in  \eqref{eq:eta-normalizations}  follows by evaluating
\eqref{eq:dual-profile-mixture-infinite} at zero.  


On the product probability space
$((0,\infty)\times\Omega,\eta\otimes\mathbb P)$ define
\begin{equation}\label{eq:J-tilde-embedding}
 \widetilde J(y,s)(\lambda,\omega)
 :=\frac{W_y(\omega)}{\sqrt{\pi\lambda}}+s.
\end{equation}
The two identities in \eqref{eq:eta-normalizations} give
integrability; indeed,
\[
 \int_{(0,\infty)}
 \E\abs{\frac{W_y}{\sqrt{\pi\lambda}}+s}\,\dd\eta(\lambda)
 \leq \norm{y}_H+\abs s.
\]
For $y\neq0$, \eqref{eq:X-dual-profile}--
\eqref{eq:dual-profile-mixture-infinite} show that
$$
 \norm{\widetilde J(y,s)}_{L_1(\eta\otimes\mathbb P)}
 =\int_{(0,\infty)}
 \E\abs{\frac{W_y}{\sqrt{\pi\lambda}}+s}
\,\dd\eta(\lambda)=\norm{y}_H\,\varphi^\circ(s/\norm{y}_H)
 =\norm{(y,s)}_{X^*}.
$$
For $y=0$, both the first and last expressions in this
calculation equal $\abs s$, because $\eta((0,\infty))=1$.  Thus
$\widetilde J$ is a linear isometric embedding of $X^*$ into an
$L_1$-space.  Since $\eta$ is
a Borel probability measure on $(0,\infty)$, the product probability
space $((0,\infty)\times\Omega,\eta\otimes\mathbb P)$ is also standard
and atomless.  By the classification theorem for standard 
probability spaces
\cite[Theorem~9.4.7, pp.~282--283]{Bogachev2007}, every atomless
standard probability space is measure-isomorphic modulo null sets to
$([0,1],\lambda)$.  Consequently, both $L_1(\Omega)$ and
$L_1((0,\infty)\times\Omega,\eta\otimes\mathbb P)$ are linearly
isometric to $L_1([0,1])$.

The norm in \eqref{eq:X-norm} is not Hilbertian.  If $e\in H$ is a
unit vector and the parallelogram identity held, its application to
$(e,0)$ and $(0,t)$ would imply
\begin{equation}\label{eq:parallelogram-phi}
 \varphi(t)^2=1+t^2,
 \qquad \mbox{ for } t\in\R.
\end{equation}
But \eqref{eq:phi-definition} gives
\begin{equation}\label{eq:phi-second-zero}
 \varphi''(0)=\frac2\pi\neq1,
\end{equation}
contradicting the second derivative of
\eqref{eq:parallelogram-phi} at zero.

Finally, because $\widetilde J:X^*\to L_1([0,1])$ is an
isometry, the Hahn--Banach theorem implies that its adjoint is a
metric quotient; see
\cite[Theorem~3.3 and Chapter~4]{Rudin1991}:
\begin{equation}\label{eq:L-infty-quotient}
 \widetilde J^{\,*}:L_\infty[0,1]\longrightarrow X^{**}.
\end{equation}
Since $X$ is reflexive by
\eqref{eq:X-Hilbert-equivalence}, $X^{**}=X$ canonically.  This proves
Theorem~\ref{thm:main-infinite-intro}.

\begin{remark}[Compatibility of the finite-dimensional family]
Choose an increasing sequence
$E_1\subset E_2\subset\cdots\subset H$ with $\dim E_j=j$ and dense
union.  For $n\geq2$, let $X_n=E_{n-1}\oplus\R$, with the restriction
of \eqref{eq:X-norm}, and identify $E_{n-1}$ orthogonally with
$\R^{n-1}$.  Then
 $B_{X_n}=Z_n^\circ$, 
 $B_{X_n^*}=Z_n$.
 Moreover, $X=\overline{\bigcup_{n\geq2}X_n}$, and the dual profile
\eqref{eq:X-dual-profile} gives the corresponding compatibility on the
dual side.  Thus the same dimension-independent Gaussian profile and positive
integral representation yield the infinite-dimensional example
directly.
\end{remark}

\section{Further remarks}\label{sec:rem}

\subsection{A remark on intersection bodies}

Recall that a star body $L\subset\R^n$ is determined by its positive
continuous radial function
\[
 \rho_L(\xi)=\max\{r\geq0:r\xi\in L\},
 \qquad \xi\in\Sph^{n-1}.
\]
Following Lutwak \cite{Lutwak1988}, the \emph{intersection body of
$L$}, denoted by $IL$, is the origin-symmetric star body whose radial
function is
\begin{equation}\label{eq:Lutwak-intersection-body}
 \rho_{IL}(\xi)
 =
 \operatorname{vol}_{n-1}(L\cap\xi^\perp)
 =
 \frac1{n-1}
 \int_{\Sph^{n-1}\cap\xi^\perp}
       \rho_L(\theta)^{n-1}\,\dd\theta .
\end{equation}
Gardner and Zhang independently considered the more general class
obtained by taking the closure of the bodies $IL$ in the radial
metric
\[
 d_r(K,M)
 =
 \max_{\xi\in\Sph^{n-1}}
 \abs{\rho_K(\xi)-\rho_M(\xi)};
\]
thus an origin-symmetric star body $K$ is called an
\emph{intersection body} if there are star bodies $L_j$ such that
$d_r(K,IL_j)\to0$ \cite{Gardner1994,Zhang1994Centered}.  We use the
term intersection body in this more general sense.

Koldobsky proved that the unit ball of every finite-dimensional
subspace of $L_1$ is an intersection body
\cite[Theorem~3]{Koldobsky1998}.  By the classical correspondence
recalled above,
\[
 Z\text{ is a zonoid}
 \quad\Longleftrightarrow\quad
 (\R^n,\norm{\cdot}_{Z^\circ})
 \text{ embeds linearly isometrically into }L_1.
\]
Consequently, the polar of every zonoid is an intersection body.
Since both $Z_n$ and $Z_n^\circ$ are zonoids, applying this implication
to each of them shows that both $Z_n$ and $Z_n^\circ$ are intersection
bodies.  Taking $K=Z_n$ and using
Theorem~\ref{thm:exact-BM}, we obtain, for every $n\geq2$,
\[
 K\text{ and }K^\circ\text{ are intersection bodies},
 \qquad
 d_{\BM}(K,B_2^n)=b_\infty^{-1}>1.
\]
For further background and equivalent analytic descriptions, see
\cite[Chapters~4--6, in particular p.~127]{Koldobsky2005}.

\subsection{A remark on the complex case}\label{complex}

After completing the real argument, we asked OpenAI's ChatGPT to
investigate the corresponding complex problem using ideas developed
in the present paper.  ChatGPT produced a complete solution.  We record this solution in the
companion arXiv note \cite{RyaboginZvavitchComplex}; no separate
journal submission is planned.  We take full responsibility for the
correctness of its statements.

The companion note shows that, for every $n\geq2$, there is a complex
zonoid $Z_n^{\mathbb C}\subset\mathbb C^n$ whose polar is also a
complex zonoid and such that
\[
 d_{\mathrm{BM}}^{\mathbb C}
 \bigl(Z_n^{\mathbb C},B_{2,\mathbb C}^n\bigr)
 =\frac{5}{2\sqrt6}>1.
\]
Here $B_{2,\mathbb C}^n$ is the Euclidean unit ball in $\mathbb C^n$, and
$d_{\mathrm{BM}}^{\mathbb C}$ denotes the Banach--Mazur distance defined
using invertible complex-linear maps. 
When regarded as bodies in $\mathbb R^{2n}$, complex zonoids are real
zonoids, and complex and real polarities agree.  For this particular
family, the proof also gives
\[
 d_{\mathrm{BM}}^{\mathbb R}
 \bigl(Z_n^{\mathbb C},B_2^{2n}\bigr)
 =\frac{5}{2\sqrt6}.
\]
Thus the complex construction also yields real counterexamples in
every even real dimension $2n \ge 4$ and already gives a negative answer to the real
convergence-to-one question.

The present paper nevertheless provides a direct construction in
every real dimension, together with a uniform $\liminf$ estimate and
an infinite-dimensional real example.

\section*{Acknowledgments}

The authors became interested in this question in 2001.  Over
the years, they discussed it with many colleagues and friends,
including Mar\'ia \'Angeles Alfonseca, Alexander Koldobsky, Joram
Lindenstrauss, Yossi Lonke, Fedor Nazarov, Mark Rudelson, Rolf
Schneider, and Vladyslav Yaskin.  They are grateful to all of them, as
well as to the many others with whom they discussed the problem, for
their interest and many valuable conversations.

The authors thank OpenAI's ChatGPT for  assistance in checking
details of the argument, improving the exposition, and preparing the
manuscript.  The authors retain full responsibility for all
mathematical claims and conclusions.

\vspace{2mm}
\noindent Dmitry Ryabogin\\
Department of Mathematical Sciences, Kent State University, Kent, OH 44242, USA.\\
E-mail address: \href{mailto:ryabogin@math.kent.edu}{ryabogin@math.kent.edu}.

\vspace{2mm}
\noindent Artem Zvavitch\\
Department of Mathematical Sciences, Kent State University, Kent, OH 44242, USA.\\
E-mail address: \href{mailto:zvavitch@math.kent.edu}{zvavitch@math.kent.edu}.

\end{document}